\documentclass[11pt, a4paper, final]{amsart}
\usepackage[T1]{fontenc}
\usepackage{lmodern}
\usepackage{mathtools}
\usepackage[utf8]{inputenc}

\usepackage{amsfonts}
\usepackage{amsmath}
\usepackage{amssymb}
\usepackage{amsthm}
\usepackage{mathtools}
\usepackage{paralist, tabularx}
\usepackage[utf8]{inputenc}
\usepackage{mathabx,epsfig}
\usepackage{tikz-cd}
\usepackage{verbatim}

\usepackage{enumitem}
\setenumerate{left=0em}

\usepackage{etoolbox}

\usepackage{xcolor}
\definecolor{green}{RGB}{0,127,0}
\definecolor{redd}{RGB}{191,0,0}
\definecolor{red}{RGB}{105,89,205}
\usepackage[colorlinks=true]{hyperref}

\usepackage[notref, notcite]{showkeys}
\usepackage[cmtip,arrow]{xy}

\DeclareMathOperator{\Aut}{Aut}
\DeclareMathOperator{\Autf}{Autf}

\newcommand{\C}{\mathfrak{C}}

\newcommand{\nref}[2]{\hyperref[#1]{\ref*{#1}$_{#2}$}}

\newcommand{\CB}{{\mathcal B}}
\newcommand{\CF}{{\mathcal F}}
\newcommand{\CP}{{\mathcal P}}
\newcommand{\CA}{{\mathcal A}}

\DeclareMathOperator{\im}{{Im}}

\DeclareMathOperator{\Th}{{Th}}

\DeclareMathOperator{\tp}{{tp}}
\DeclareMathOperator{\cl}{{cl}}
\DeclareMathOperator{\dcl}{{dcl}}
\DeclareMathOperator{\Inv}{Inv}
\DeclareMathOperator{\id}{id}
\DeclareMathOperator{\Gal}{Gal}
\DeclareMathOperator{\ev}{ev}
\DeclareMathOperator{\Age}{Age}
\DeclareMathOperator{\Def}{Def}
\DeclareMathOperator{\Extdef}{Extdef}

\def\invlim{\underleftarrow{\lim}}

\newtheorem{theorem}{Theorem}
\numberwithin{theorem}{section}
\newtheorem{lemma}[theorem]{Lemma}
\newtheorem{fact}[theorem]{Fact}
\newtheorem{proposition}[theorem]{Proposition}
\newtheorem{problem}[theorem]{Problem}

\newtheorem{question}[theorem]{Question}
\newtheorem{corollary}[theorem]{Corollary}

\newtheorem{clm}{Claim}
\newtheorem*{clm*}{Claim}

\newenvironment{clmproof}{
    
    \begin{proof}
    
}{
    \end{proof}
}

\theoremstyle{definition}
\newtheorem{definition}[theorem]{Definition}

\newtheorem{example}[theorem]{Example}

\theoremstyle{remark}
\newtheorem{remark}[theorem]{Remark}

\AtEndEnvironment{theorem}{\setcounter{clm}{0}}
\AtEndEnvironment{proposition}{\setcounter{clm}{0}}
\AtEndEnvironment{lemma}{\setcounter{clm}{0}}
\AtEndEnvironment{corollary}{\setcounter{clm}{0}}
\AtEndEnvironment{remark}{\setcounter{clm}{0}}

\usepackage[hyphenbreaks]{breakurl}

\usepackage[backend=bibtex,  sorting=nyt, style=alphabetic]{biblatex}
\usepackage{orcidlink}

\title{Ramsey theory and topological dynamics of 0-dimensional flows}

\author{Krzysztof Krupi\'{n}ski\ \orcidlink{0000-0002-2243-4411}}
\address{\parbox{\linewidth}{Instytut Matematyczny, Uniwersytet Wroc{\l}awski\\
pl. Grunwaldzki 2, 50-384 Wroc{\l}aw, Poland}}
\email{Krzysztof.Krupinski@math.uni.wroc.pl}

\author{Junguk Lee\ \orcidlink{0000-0002-2150-6145}}
\address{\parbox{\linewidth}{Department of Mathematics, Changwon National University\\
20 Changwondaehak-ro, Uichang-gu, Changwon-si, Gyeongsangnam-do, 51140, South Korea}}
\email{ljwhayo@changwon.ac.kr}

\author{Slavko Moconja\ \orcidlink{0000-0003-4095-8830}}
\address{\parbox{\linewidth}{University of Belgrade, Faculty of Mathematics\\
Studenski trg 16, 11158 Belgrade, Serbia}}
\email{slavko@matf.bg.ac.rs}

\thanks{\noindent The first author was supported by the Narodowe Centrum Nauki grants no.\ 2016/22/E/ST1/00450. The second author was supported by the National Research Foundation of Korea (NRF) grant funded by the Korea government (MSIT) (RS-2024-00451678). The third author was partially supported by the Ministry of  Science, Technological Development and Innovation, Republic of Serbia, through the project 451-03-33/2026-03/200104.}

\keywords{Definable Ramsey property, separately finite definable Ramsey degree, topological dynamics, 0-dimensional ambit, Ellis semigroup, Ellis group}
\subjclass[2020]{5D10, 37B02, 54H11, 20E18, 03C98, 03C45}

\begin{document}

\maketitle

\begin{abstract}
We introduce several Ramsey-theoretic properties of $0$-dimensional ambits and obtain their dynamical characterizations. In consequence, we obtain Ramsey-theoretic criteria for triviality and for profiniteness of some important invariants (in particular, of the Ellis groups) of the ambits in question. This yields a criterion for the structural property that each distal minimal factor of the given ambit is a profinite flow. Another result is a criterion for metrizability of the minimal left ideals in the Ellis semigroup. 

Then we study three specializations of the above abstract context: to first order theories, to definable groups, and to the classical Kechris-Pestov-Todor\v{c}evi\'{c} theory, recovering known and obtaining new results in each of these contexts.
\end{abstract}

\section{Introduction}

Structural Ramsey theory is an important field of research in combinatorics originated in the 1970s with fundamental contributions of Ne\v{s}et\v{r}il and R\"{o}dl \cite{NeRo}, which concerns colorings of copies of a finite structure inside another structure (often from a given Fra\"{i}ss\'e class). Since the early 2000s structural Ramsey theory has been seeing a flaring surge in interest, ignited by the milestone contribution of Kechris, Pestov, and Todor\v{c}evi\'{c} \cite{KPT}, connecting structural Ramsey theory to topological dynamics by discovering correspondences between Ramsey-theoretic properties of Fra\"{i}ss\'{e} classes and dynamical properties of the automorphism groups of their Fra\"{i}ss\'{e} limits. This inspired a whole new school and made the subject interesting not only in combinatorics but also in topological dynamics and descriptive set theory. KPT will be an abbreviation for Kechris-Pestov-Todor\v{c}evi\'{c}.

Our paper \cite{KLM} brought to the picture also model theory, yielding several correspondences between definable Ramsey-theoretic and dynamical properties of first order theories, resulting in combinatorial criteria for triviality and [pro]finiteness of important invariants in model theory, such as the Kim-Pillay Galois groups of first order theories, which in turn yields such criteria for equality of Kim-Pillay strong types, Shelah strong types, and usual types.

The main idea and goal of the present paper is to generalize the concept of definable Ramsey properties from \cite{KLM} to arbitrary $0$-dimensional ambits, proving correspondences between those kinds of properties and appropriate dynamical properties of the underlying ambits, finding further consequences, and showing that some of our ``definable'' Ramsey properties follow from suitable general dynamical assumptions on the underlying flows.
Among the main contributions, we get a Ramsey-theoretic characterization of profiniteness of the Ellis group of a $0$-dimensional ambit; notably each distal minimal flow whose Ellis group is profinite is a profinite flow (i.e.\ an inverse limit of finite flows), so this yields a Ramsey-theoretic criterion for profiniteness of the distal minimal factors of the given $0$-dimensional ambit. We also find a Ramsey-theoretic criterion for metrizability of minimal left ideals in the Ellis semigroup of a $0$-dimensional ambit.

The general context of $0$-dimensional ambits specializes to the context of first order theories from \cite{KLM} (mostly recovering the various notions and results from \cite{KLM}, but also yielding new ones), to definable groups acting on associated spaces of types (yielding new notions, results, and questions), and to the classical KPT context (yielding several known and some new notions, results, and questions). 

There have been many natural questions concerning relationships between the various properties that we introduce, both in the abstract context of $0$-dimensional ambits and in the above three specializations. This leads to some implications or equivalences but also to some interesting counterexamples. Some of our basic questions remain open, most of them with expected negative answers which should be obtained by constructing suitable counterexamples. Some of these questions will be stated in this paper.



\subsection{Results and structure of the paper}
Section \ref{section: main} contains the main results of this paper. It concerns arbitrary $0$-dimensional ambits and does not involve any model theory.

Let $(G,X,x_0)$ be a $0$-dimensional ambit and $S$ be a subset of $X$. By $E(X)$ we denote the Ellis semigroup of $X$.

In Subsection \ref{subsection: 3.1}, we introduce the notion of {\em $S$-definable Ramsey property} ($S$-DRP) for the ambit $(G,X,x_0)$, and we obtain its dynamical characterization: there exists an element $\eta \in E(X)$ such that $\eta[S]$ consists of invariant elements (see Definition \ref{definition: S-DERP} and Proposition \ref{prop: characteriazation_S-derp}). This is an adaptation to the abstract context of a similar result in the context of first order theories from \cite{KLM}. The most interesting cases are when $S=X$, $S=Gx_0$, and $S=\{x_0\}$. The corresponding $S$-DRP are called the {\em externally definable  Ramsey property} (EDRP), {\em definable  Ramsey property} (DRP), and {\em weakly definable  Ramsey property} (WDRP), respectively. We note that WDRP is equivalent to the existence of a fixed point in $X$. It follows from the definition that EDRP implies DRP implies WDRP. By \cite[Example 5.7]{KLM} and explanations in Subsection \ref{subsection: first order theories}, we know that that EDRP is strictly stronger that DRP. In Subsection \ref{subsection: 3.1}, we give criteria for EDRP to be equivalent to DRP and also for DRP to be equivalent to WDRP. We give an example showing that in general WDRP does not imply DRP.

In Subsection \ref{subsubsection: Ramsey and profiniteness}, we introduce the notion of $(G,X,x_0)$ having {\em separately finite $S$-definable  Ramsey degree} (sep.\ fin.\ $S$-DRdeg). When $S=X$, we call it {\em separately finite externally definable  Ramsey degree} (sep.\ fin.\ EDRdeg); for $S=Gx_0$ we call it {\em sep.\ fin.\ DRdeg}. One of the main results is the dynamical characterization of this property obtained in Theorem \ref{theorem: characterization_sep_fin_S-DERdeg} (see Corollary \ref{corollary: characterization_sep_fin_EDERdeg} in the case $S=X$). This implies Corollary \ref{corollary: Ellis group is profinite}, which tells us that sep.\ fin.\ EDRdeg implies that the Ellis group of the ambit $(G,X,x_0)$ is profinite, which in turn implies, using some fundamental results in topological dynamics, that every distal minimal flow which is a factor of $(G,X,x_0)$ is a profinite flow (see Corollary \ref{corollary: sep. fin. EDERdeg implies profiniteness of distal factors}). This is a structural consequence of having sep.\ fin.\ EDRdeg. Then we state several natural conditions between sep.\ fin.\ EDRdeg and profiniteness of the {\em Bohr compactification of $G$ below the ambit $(G,X,x_0)$} (see Definition \ref{definition: Bohr compactification below ambit} and Proposition \ref{proposition: relationships}). Proposition \ref{proposition: relationships} is a straightforward adaption to the abstract context of an analogous fact for first order theories from \cite{KLM}, except the Bohr compactifications which will be dealt with in a forthcoming paper. We discuss the missing implications between the aforementioned conditions. Some of them are left as questions, but, notably, in Section \ref{section: questions and examples}, we give an example showing that in the general abstract context condition (A) in Proposition \ref{proposition: relationships} does not imply (B); in the context of both first order theories and definable groups, this implication remains an open problem.

While most of the notions and results in Subsections \ref{subsection: 3.1} and \ref{subsubsection: Ramsey and profiniteness} are motivated by the corresponding considerations in \cite{KLM}, Subsection \ref{subsection: joint. fin. EDERdeg} is motivated directly by the classical notion of finite Ramsey degree of a Fra\"{i}ss\'{e} class and Zucker's characterization from \cite[Theorem 8.7]{Zuc} of metrizability of the universal minimal $\Aut(M)$-flow for a  Fra\"{i}ss\'{e} structure $M$ (see also \cite[Theorem 6.3]{KrPi2}). We introduce the notion of $(G,X,x_0)$ having {\em jointly finite $S$-definable  Ramsey degree} (joint.\ fin.\ $S$-DRdeg) and we prove its dynamical characterization which says that some/every minimal left ideal in $E(X)$ is finite. Next, we study a related but more delicate issue. For a fixed collection $\Sigma$ of clopens in $X$, we introduce the notion of {\em joint.\ fin.\ EDRdeg with respect to $\Sigma$} (see Definition \ref{definition: joint. fin. EDERdeg with respect to Sigma}), and, using it, in Lemma \ref{lemma: main lemma for metrizability} we give a criterion for a natural presentation of some left ideal in $E(X)$ as an inverse limit of finite spaces. This yields a criterion for metrizability of the minimal left ideals in $E(X)$ stated in Corollary \ref{corollary: criterion for metrizability}.

In Subsection \ref{subsection: examples}, we start from Proposition \ref{proposition: WAP implies finite RP} saying that each weakly almost periodic (WAP) $0$-dimensional ambit has sep.\ fin.\ EDRdeg. Then we find a condition for a Boolean $G$-algebra (i.e.\ closed under the left translations) $\mathcal{A}$ of subsets of a group $G$ which guarantees that $(G,S(\mathcal{A}),p_e)$ does not have sep.\ fin.\ DRdeg (where $S(\mathcal{A})$ is the Stone space of $\mathcal{A}$ and $p_e$ is the principal ultrafilter at $e$). Using it, we obtain Proposition \ref{proposition: No sep. fin DERdeg} which says that whenever $G$ is a (discrete) infinite abelian group, the universal $G$-ambit $(G,\beta G, p_e)$ does not have sep.\ fin.\ DRdeg. We state as a question whether abelianity could be dropped in this result which would be a strengthening of the well-known result that all discrete groups are not extremely amenable. Another question is whether the notions of having sep.\ fin.\ EDRdeg and DRdeg are equivalent.

Section \ref{section: specializations} analyzes three specializations of the abstract context developed in Section \ref{section: main}.

In Subsection \ref{subsection: first order theories}, we apply the abstract context to first order theories, i.e.\ to the $0$-dimensional ambit $(\Aut(\C),S_{\bar{c}}(\C),\tp(\bar c/\C))$, where $\C$ is a monster model of a given first order theory $T$, $\bar c$ is its enumeration, and $S_{\bar{c}}(\C)$ is the space of all complete types over $\C$ in variables $\bar x$ corresponding to $\bar c$ which extend $\tp(\bar c/\emptyset)$. We discuss a natural correspondence between definable colorings of elements of $\Aut(\C)$ in the sense of this paper and definable colorings of copies of finite tuples in $\C$ in the sense of \cite{KLM}. This yields a correspondence between the relevant Ramsey properties (see Proposition \ref{proposition: specialization of EDERP to theories}). All of this justifies our terminology involving [external] definability, shows that the results from Subsections \ref{subsection: 3.1} and \ref{subsubsection: Ramsey and profiniteness} imply the corresponding main results in \cite{KLM}, and provides lots of examples of ambits with the various Ramsey properties introduced in Section \ref{section: main}. On the other hand, Subsection \ref{subsection: joint. fin. EDERdeg} specializes to new notions and results for first order theories. Moreover, in Proposition \ref{proposition: under NIP DERP iff EDERP for theories}, we show that if $T$ has NIP, then EDRP is equivalent to DRP, answering a question asked by Itay Kaplan. Finally, Theorem \ref{theorem: improved Meir Sullivan} generalizes \cite[Theorem 5.1]{MeSu} to any uncountable structure $M$: $M$ has the externally definable Ramsey property (see the paragraph before Fact \ref{fact: Meur and Sullivan}) iff every $\Aut(M)$-subflow of every $S_n(M)$ has a fixed point. This result is a simple consequence of Corollary \ref{cor: amenability_derp}(3).

In Subsection \ref{subsection: definable groups}, we apply the abstract context to definable groups, i.e.\  to the $0$-dimensional ambits $(G,S_G(M),\tp(e/M))$ and $(G,S_{G,ext}(M),\tp_{ext}(e/M))$, where $G$ is a group definable in a structure $M$, $S_G(M)$ is the space of complete types over $M$ concentrated on $G$, and $S_{G,ext}(M)$ is the space of complete external types over $M$ concentrated on $G$ (i.e.\ ultrafilters on the Boolean algebra of externally definable subsets of $G$). In this context, Section \ref{section: main} specializes to new notions and theorems. Besides that we observe that stability implies that those two ambits have sep.\ fin.\ EDRdeg (note that they coincide by stability). Next, in Corollary \ref{corollary: DEERP iff EDERP under NIP for definable groups}, we get that if the theory of $M$ has NIP, then for each of the above two ambits the properties EDRP and DRP are equivalent. In Corollary \ref{corollary: equivalences of RP for definable groups}, we observe that (without any assumptions on $M$) for the ambit $(G,S_{G,ext}(M),\tp_{ext}(e/M))$ all three properties WDRP, DRP, EDRP are equivalent. For the ambit $(G,S_G(M),\tp(e/M))$, WDRP is equivalent to DRP, but Example \ref{example: DERP but not EDERP fr definable groups} shows that DRP does not imply EDRP. We also give an example showing that EDRP for $(G,S_G(M),\tp(e/M))$ does not imply EDRP for $(G,S_{G,ext}(M),\tp_{ext}(e/M))$, and we ask if the implication is true under NIP.

In Subsection \ref{subsection: KPT context}, we apply the abstract context to the classical KPT context, i.e.\ to the universal $\Aut(M)$-ambit $(\Aut(M),\mathcal{U},p_0)$, where $M$ is any first order structure and $\Aut(M)$ is its group of automorphisms equipped with the pointwise convergence topology. We use a very convenient model-theoretic description of this ambit from \cite[Theorem 2.2]{KrPi2}. We observe that each of the properties WDRP, DRP, EDRP coincides with $M$ having the embedding Ramsey property in the classical sense (see Definition \ref{definition: Ramsey property}(1)). Remark \ref{remark: sep. fin. ERdeg vs joint. fin. EDErdeg} recovers the classical notion of $M$ having sep.\ fin.\ embedding Ramsey degree (see Definition \ref{definition: Ramsey property}(2)) in terms of joint.\ fin.\ EDRdeg of $(\Aut(M),\mathcal{U},p_0)$ with respect to $\Sigma_i$, for all $i \in I$ for a certain canonical collection $\{\Sigma_i\}_{i\in I}$ of families of clopens in $\mathcal{U}$. Using this remark, Zucker's criterion for metrizability of the universal minimal $\Aut(M)$-flow follows from Corollary \ref{corollary: criterion for metrizability}. The main new notion from Section \ref{section: main} is sep.\ fin.\ EDRdeg for a $0$-dimensional ambit. Applied to $(\Aut(M),\mathcal{U},p_0)$, it is weaker than $M$ having sep.\ fin.\ embedding Ramsey degree and it seems to be strictly weaker, although we do not know any example confirming strictness. The specializations of the dynamical characterization from Corollary \ref{corollary: characterization_sep_fin_EDERdeg} and of the main consequences obtained in Corollaries \ref{corollary: Ellis group is profinite}, \ref{corollary: sep. fin. EDERdeg implies profiniteness of distal factors} are new results in the KPT context. An interesting summary is contained in Corollary \ref{corollary: main new result in KPT context}: sep.\ fin.\ EDRdeg of the universal ambit  $(\Aut(M),\mathcal{U},p_0)$ implies that every distal minimal $\Aut(M)$-flow is a profinite flow; in particular, if $M$ has sep.\ fin.\ embedding Ramsey degree, then every distal minimal $\Aut(M)$-flow is a profinite flow.

Section \ref{section: questions and examples} summarizes some aspects and discusses various natural questions. It also provides a counterexample to one of the missing implications in Proposition \ref{proposition: relationships} which was already mentioned above.

\section{Preliminaries}


We recall some notions and facts from topological dynamics. Regarding Ramsey theory, all the relevant notions will be defined in suitable places in the main body of the paper.


A {\em flow} is a pair $(G,X)$ where $G$ is a topological group acting continuously on a compact space $X$. An {\em ambit} is a flow $(G,X,x_0)$ with a distinguished point $x _0 \in X$ with dense $G$-orbit. If $G$ is discrete, the action being continuous just means that it is an action by homeomorphisms. A flow $(G,Y)$ is said to be a {\em factor} of the flow $(G,X)$ if there exists a $G$-flow epimorphism from $(G,X)$ to $(G,Y)$.

We discuss fundamental notions and results from Ellis theory. For the proofs see e.g.\ \cite{Gla} or \cite[Appendix A]{Rze}.

Let $(G,X)$ be a flow. Working in the space $X^X$ with the product topology, 
by identifying $g\in G$ with the function $X\to X$ given by $x\mapsto gx$ (by abusing notation, also denoted by $g$), one sees $G$ as a subset of $X^X$. The closure of $G$ forms a semigroup with composition as the semigroup operation. 
This semigroup is called the {\em  Ellis semigroup} of $X$ and is denoted by $E(X)$. The semigroup operation on $E(X)$ is continuous in the left coordinate but in general not in the right coordinate. $E(X)$ has a natural structure of a $G$-ambit with $\id_X$ as the distinguished point with dense orbit.

A minimal left ideal in $E(X)$ always exists and is the same thing as a minimal $G$-subflow of $E(X)$. Each minimal left ideal $\mathcal{M}$ in $E(X)$ contains at least one idempotent (an element $u$ such that $u^2=u$). In fact, $\mathcal{M}$ is the disjoint union of the sets $u\mathcal{M}$ for $u$ ranging over the idempotents in $\mathcal{M}$. Each subset $u\mathcal{M}$ is a group with respect to the restriction of the semigroup operation on $E(X)$. All these groups $u\mathcal{M}$ are isomorphic and the common isomorphism type of them (or any of these groups itself) will be called the {\em Ellis group of $X$}. 

On any Ellis group $u\mathcal{M}$, there is the so-called {\em $\tau$-topology} which is weaker and much more useful than the topology induced from $E(X)$. First, for $a\in E(X)$ and $B\subseteq E(X)$ let $a\circ B$ be the set of all limits of the nets $(g_ib_i)_i$ such that $g_i\in G$, $b_i\in B$, and $\lim_ig_i=a$. For $B\subseteq u\mathcal{M}$ we define $\cl_\tau(B)= u\mathcal{M}\cap(u\circ B)$.
$\cl_\tau$ is a closure operator on $u\mathcal{M}$; the {\em $\tau$-topology} is the topology on $u\mathcal{M}$ induced by $\cl_\tau$. The group $u\mathcal{M}$ with the $\tau$-topology is a quasi-compact,  $T_1$ semitopological group (i.e., group operation is separately continuous). The isomorphism types of the Ellis groups (for all minimal $\mathcal{M}$ and idempotents $u \in \mathcal{M}$) as semitopological groups do not depend on the choice of $u$ and $\mathcal{M}$.  Put $H(u\mathcal{M})=\bigcap_U\cl_\tau(U)$, where the intersection is taken over all $\tau$-open neighborhoods of $u$ in $u\mathcal{M}$. This is a $\tau$-closed normal subgroup of $u\mathcal{M}$, and the quotient $u\mathcal{M}/H(u\mathcal{M})$ is a compact, Hausdorff topological group. Moreover, $H(u\mathcal{M})$ is the smallest $\tau$-closed normal subgroup of $u\mathcal{M}$  such that $u\mathcal{M}/H(u\mathcal{M})$ is Hausdorff.

Now, we will state a few more specific results which will be useful in this paper.

\begin{remark}\label{remark: charact. of Inv(X)}
Suppose that $X$ is in addition totally disconnected.
Let $\Inv(X)$ be the set of fixed points in $X$. Then, for any $x\in X$, $x\in \Inv(X)$ if and only if for every clopen subset $U$ of $X$: $$x\in U\Leftrightarrow x\in\bigcap_{g\in G} gU.$$
\end{remark}
\begin{proof}
It is enough to show ($\Leftarrow$). Take $x\in X$. Suppose there exists $g\in G$ such that $gx\neq x$. Since $X$ is a profinite space, there exists a clopen subset $U$ of $X$ such that $x\in U$ and $gx\in U^c$ so that $x\in U\cap g^{-1}U^c$, which completes the proof.
\end{proof}

Let $\mathcal{A}$ be a Boolean $G$-algebra (shortly, $G$-algebra) of subsets of $G$, i.e.\ a Boolean algebra closed under the left translations by the elements of $G$. 
By $S(\mathcal{A})$ we denote the Stone space of $\mathcal{A}$. The group $G$ acts naturally on $S(\mathcal{A})$ by left translations, turning $S(\mathcal{A})$ into a $G$-ambit with the principal ultrafilter $p_e:=\{A \in \mathcal{A}: e \in A\}$ as the distinguished point. 
For $p \in S(\mathcal{A})$ the operator $d_p \colon \mathcal{A} \to \mathcal{P}(G)$ is defined by: 
$$d_p(A):=\{ g\in G: g^{-1}A \in p\}.$$ 
Each $d_p$ is a homomorphism of Boolean $G$-algebras. The algebra $\mathcal{A}$ is called {\em d-closed} if $d_p[\mathcal{A}]=\mathcal{A}$. The next fundamental fact essentially comes from Newelski (cf. \cite[Proposition 2.4]{Ne}); a full proof of all equivalences appeared in Micha\l\  Baran's Master's thesis written under the supervision of the first author; we leave it as an exercise.

\begin{fact}\label{fact: d-closed iff semigroup}
Let $\mathcal{A}$ be a $G$-algebra of subsets of $G$. The following conditions are equivalent:
\begin{enumerate}
\item $\mathcal{A}$ is d-closed.
\item There exists a left-continuous semigroup operation $*$ on $S(\mathcal{A})$ extending the action of $G$ (in the sense that for every $g \in G$ and $q \in S(\mathcal{A})$ we have $p_g * q=gq$, where $p_g:=\{A \in \mathcal{A}: g \in A\}$).
\item There exists a semigroup operation $*$ on $S(\mathcal{A})$ such that the map $l \colon S(\mathcal{A}) \to E(S(\mathcal{A}))$ given by $l(p)(q):=p*q$ is a topological isomorphism of semigroups which satisfies $l(p_g)(q)=gq$ for all $g \in G$ and $q \in S(\mathcal{A})$. (Then $*$ is left-continuous, extends the action of $G$, and $l$ is also an isomorphism of $G$-flows.)
\end{enumerate}
Moreover, in each of the items (1) and (2) we have uniqueness of $*$ with the required properties, and $*$ in (1) and (2) coincide.
\end{fact}

The next fact was observed by Newelski, but appeared in Adam Malinowski's thesis \cite{Mal}.

\begin{fact}\cite[Fact 2.22 and Theorem 2.24]{Mal}\label{fact: A^d}
Let $\mathcal{A}$ be a Boolean $G$-algebra of subsets of $G$. There exists a smallest d-closed $G$-algebra of subsets of $G$ containing $\mathcal{A}$ which is denoted by $\mathcal{A}^d$. The d-closure $\mathcal{A}^d$ is realized in one step, i.e.\ $\mathcal{A}^d$ is precisely the Boolean algebra generated by $\{d_p(A): A \in \mathcal{A}, p \in S(\mathcal{A})\}$. Moreover, $E(S(\mathcal{A})) \cong S(\mathcal{A}^d)$ as semigroups and as $G$-flows (with respect to the unique left-continuous semigroup operation $*$ on $S(\mathcal{A}^d)$ extending the action of $G$ provided by Fact \ref{fact: d-closed iff semigroup}).
\end{fact}

The following fact is folklore. It shows that while considering 0-dimensional ambits, without loss of generality we can just study the ambits of the form $(G,S(\mathcal{A}),p_e)$ where $\mathcal{A}$ is a $G$-algebra of subsets of $G$ and $p_e:=\{A \in \mathcal{A}: e \in A\}$.

\begin{fact}\label{fact: 0-dim ambit is S(A)}
Let $(G,X,x_0)$ be a $0$-dimensional ambit. Let $\mathcal{A}$ be the Boolean $G$-algebra of subsets $G$ consisting of the sets of the form $U_G:=\{g \in G: gx_0 \in U\}$ for $U$ ranging over the clopen subsets of $X$, and $p_e:=\{A \in \mathcal{A}: e \in A\}$. Then the map $\Phi \colon X \to S(\mathcal{A})$ given by $\Phi(x):=q_x$ where $q_x:=\{U_G: x \in U \subseteq X,\, U \textrm{ clopen}\}$ is an isomorphism between the ambits $(G,X,x_0)$ and $(G,S(\mathcal{A}),p_e)$.
\end{fact}


\section{The main results in the abstract context}\label{section: main}

In this section, we introduce various notions of definable Ramsey properties and having finite Ramsey degrees for $0$-dimensional ambits, and we find their dynamical characterizations and consequences.

From now on, let $(G,X,x_0)$ be a $0$-dimensional ambit. 
By $\CB$ we denote the Boolean $G$-algebra of all clopen subsets of $X$.

\subsection{Definable colorings and Ramsey properties}\label{subsection: 3.1}

The goal of this subsection is to introduce variants of definable Ramsey properties, prove their dynamical characterizations (yielding, in particular, a characterization of triviality of all minimal subflows of $(G,X)$), and study the relationships between the introduced variants of Ramsey properties.

 
\begin{definition}\label{definition: definable colorings}
A {\em coloring} is a function $c \colon G \to 2^n$ for some natural number $n \geq 1$. Let $c \colon G \to 2^n$ be a coloring.
\begin{enumerate}
	\item Let $S$ be a subset of $X$. Then $c$ is called {\em $S$-definable} if there exist $U_0,\ldots, U_{n-1}\in \CB$ and $y_0,\ldots,y_{n-1}\in S$ such that for any $g\in G$ and $i<n$,
$$c(g)(i)=
\begin{cases}
1 & \mbox{ if }y_i\in gU_i\\
0 & \mbox{ if } y_i\notin gU_i
\end{cases}.$$

	\item $c$ is called {\em strongly definable} if it is $\{x_0\}$-definable (shortly, $x_0$-definable).

	\item $c$ is called {\em definable} if it is $Gx_0$-definable.

	\item $c$ is is called {\em externally definable} if it is $X$-definable.
\end{enumerate}
\end{definition}

The above terminology is adapted from \cite{KLM}. We will see in Subsection \ref{subsection: first order theories} that in the context of theories 
 the above notions boil down to the usual notions of [externally] definable map from a definable or rather type-definable set to a compact (in fact, finite) space.

\begin{definition}\phantomsection\label{definition: S-DERP}
\begin{enumerate}
	\item Let $S$ be a subset of $X$. The $G$-ambit $X$ has {\em $S$-definable  Ramsey property} ($S$-DRP) if for any $S$-definable coloring $c\colon G\rightarrow 2^n$, for all $k\ge 1$, and for all $A\in [G]^k$, there exists $h\in G$ such that $|c[hA]|=1$.

\item The $G$-ambit $X$ has {\em weakly definable  Ramsey property} (WDRP), {\em definable  Ramsey property} (DRP), {\em externally definable  Ramsey property} (EDRP) if it has $\{x_0\}$-, $Gx_0$-, $X$-DRP, respectively.
\end{enumerate}
\end{definition}

The next proposition is a generalization of Theorem 3.14 from \cite{KLM} (see Section \ref{section: specializations}) and the proof below is an adaptation of the idea of the proof of that theorem to the abstract context.

\begin{proposition}\label{prop: characteriazation_S-derp}
Let $S\subseteq X$. The following are equivalent:
\begin{enumerate}
	\item The $G$-ambit $X$ has $S$-DRP.
	\item There exists $\eta\in E(X)$ such that $\eta[S]\subseteq \Inv(X)$.
\end{enumerate}
\end{proposition}

\begin{proof}
($\Rightarrow$) Assume that $X$ has $S$-DRP.

\begin{clm}\label{clm: Claim 1 in S-DERP}
For $\bar U=(U_0,\ldots,U_{n-1})\in \CB^n$, $\bar y=(y_0,\ldots,y_{n-1})\in S^n$, and $A=\{g_0,\ldots,g_{k-1}\}\in [G]^k$, there exists $g^*=g_{\bar U,\bar y,A}\in G$ such that for all $i<n$ and $j<k$, $$g^*y_i\in g_0U_i\Leftrightarrow g^*y_i\in g_jU_i.$$
\end{clm}

\begin{clmproof}
Consider the $S$-definable coloring $c=c_{\bar U,\bar y} \colon G\rightarrow 2^n$ defined as follows: For $g\in G$ and $i<n$, $$c(g)(i):=
\begin{cases}
1 & \mbox{ if } y_i\in gU_i\\
0 & \mbox{ if } y_i\notin gU_i
\end{cases}.$$
By $S$-DRP, there exists $h\in G$ such that $|c[hA]|=1$. Then, $g_{\bar U,\bar y,A}:=h^{-1}$ satisfies the requirements.
\end{clmproof}
Define a pre-order $\le$ on $\CF:=\bigcup_{n,k} \CB^n\times S^n\times [G]^k$ as follows: For $(\bar U,\bar y, A)\in \CB^n\times S^n\times [G]^k$ and $(\bar U',\bar y',A') \in \CB^{n'}\times S^{n'}\times [G]^{k'}$, 
$$(\bar U,\bar y, A)\le (\bar U',\bar y',A')$$
iff $n\leq n'$, $A \subseteq A'$,  and $\{(U_i,y_i)\}_{i<n} \subseteq \{(U'_j,y'_j)\}_{j<n'}$. Then, $\CF$ is directed by $\le$. 

Consider a net $(g_{\bar U,\bar y,A})_{(\bar U,\bar y,A)\in \CF}$ of elements of $G$ obtained by Claim \ref{clm: Claim 1 in S-DERP}. 
Let $\eta \in E(X)$ be the limit of a convergent subnet.

\begin{clm}
$\eta[S]\subseteq \Inv(X)$.
\end{clm}

\begin{clmproof}
Suppose there exists $y\in S$ such that $\eta(y)\notin \Inv(X)$. Then, by Remark \ref{remark: charact. of Inv(X)}, there exist a clopen subset $U$ of $X$ and $g\in G$ such that $\eta(y)\in U\cap gU^c$. Since $U\cap gU^c$ is open, there exists $(\bar U, \bar y, A)\in \CF$ such that:
\begin{enumerate} 
\item $U_i=U$ and $y_i=y$ for some $i$,
\item $e,g\in A$,
\item for $g^*:=g_{\bar U,\bar y, A}$ we have $g^*y\in U\cap gU^c$. 
\end{enumerate}
By (1), (2), and the fact that $g^*=g_{\bar U,\bar y, A}$ satisfies the conclusion of Claim \ref{clm: Claim 1 in S-DERP}, we get 
$$g^*y\in U\Leftrightarrow g^*y\in gU,$$ 
which contradicts item (3).
\end{clmproof}

($\Leftarrow$) Suppose there exists $\eta\in E(X)$ such that $\eta[S]\subseteq \Inv(X)$. Let $c\colon G\rightarrow 2^n$ be an $S$-definable coloring given by $U_0,\ldots,U_{n-1}\in \CB$ and $y_0,\ldots,y_{n-1}\in S$. Consider any $A\in [G]^k$. Since $\eta(y_i)\in \Inv(X)$, by Remark \ref{remark: charact. of Inv(X)}, we have that $$\bigwedge_i \left( \eta(y_i)\in \bigcap_{g\in A}gU_i \cup \bigcap_{g\in A}gU_i^c\right),$$ which is an open condition for $\eta$. So, there exists $h\in G$ such that $$\bigwedge_i \left( h^{-1}y_i\in \bigcap_{g\in A}gU_i \cup \bigcap_{g\in A}gU_i^c\right),$$ which implies that $|c[hA]|=1$.
\end{proof}

\begin{corollary}\label{cor: characterization_wderp}
The following are equivalent:
\begin{enumerate}
	\item The ambit $X$ has WDRP.
	\item $\Inv(X)$ is non-empty.
\end{enumerate}
\end{corollary}

\begin{proof}
The implication $(1)\Rightarrow (2)$ follows from Proposition \ref{prop: characteriazation_S-derp}. It remains to prove $(2)\Rightarrow (1)$. Suppose $\Inv(X)\neq \emptyset$. Let $c\colon G\rightarrow 2^n$ be a strongly definable coloring given by $U_0,\ldots,U_{n-1}\in \CB$. Consider any $A\in [G]^k$. Let $y\in \Inv(X)$. Then, by Remark \ref{remark: charact. of Inv(X)}, we have that 
$$y\in \bigcap_{g\in A}gU_i\cup \bigcap_{g\in A}gU_i^c,\ \text{for every }i<n.$$ Since $Gx_0$ is dense in $X$, there exists $h\in G$ such that $$h^{-1}x_0\in \bigcap_{g\in A}gU_i\cup \bigcap_{g\in A}gU_i^c,\ \text{for every }i<n.$$ and so $$x_0\in \bigcap_{g\in A}hgU_i\cup \bigcap_{g\in A}hgU_i^c,\ \text{for every }i<n.$$ which is equivalent to $|c[hA]|=1$.
\end{proof}

\begin{remark}\label{remark: equivalent dynamical characterizations of EDERP}
The following conditions are equivalent:
\begin{enumerate}
	\item There exists $\eta\in E(X)$ such that $\im(\eta)\subseteq \Inv(X)$.
	\item There exists a $G$-invariant element $\eta \in E(X)$.
	\item Every/some minimal left ideal in $E(X)$ is trivial (i.e.\ singleton).
	\item Every minimal subflow of $X$  is trivial.
\end{enumerate}
\end{remark}

\begin{proof}
The equivalence (1) $\Leftrightarrow$ (2) is trivial with the same $\eta$ in (1) and (2). 

The implication (3) $\Rightarrow$ (2) is trivial. For the converse consider any minimal left ideal $\mathcal{M}$ in $E(X)$, take $\eta_0 \in \mathcal{M}$ and any $G$-invariant $\eta \in E(X)$. Then $\eta \eta_0 \in \mathcal{M}$ is $G$-invariant, so $\mathcal{M}$ is trivial by minimality.

(1) $\Rightarrow$ (4). Take $\eta \in E(X)$ with $\im(\eta)\subseteq \Inv(X)$. Consider any $x \in X$ in any minimal subflow $M$ of $X$; note that $M=\cl(Gx)$. Then $\eta(x) \in \cl(Gx)$ and $G\eta(x)=\{\eta(x)\}$, so $M=\cl(Gx) =\{\eta(x)\} =\{x\}$ by minimality of $M$.

(4) $\Rightarrow$ (1) Assume that every minimal subflow of $X$ is trivial. Consider any $\eta \in E(X)$ which belongs to a minimal left ideal. Then for every $x \in X$ the element $\eta(x) \in \cl(Gx)$ belongs to a minimal subflow of $X$ and so $\eta(x) \in \Inv(X)$.
\end{proof}

Summarizing, we have:

\begin{corollary}\label{cor: amenability_derp}
Let $(G,X,x_0)$ be a $0$-dimensional ambit.
\begin{enumerate}
	\item The following are equivalent:
	\begin{enumerate}
		\item The $G$-ambit $X$ has WDRP.
		\item $\Inv(X)\neq \emptyset$.
	\end{enumerate}
	
	\item The following are equivalent:
	\begin{enumerate}
		\item The $G$-ambit $X$ has DRP.
		\item There exists $\eta\in E(X)$ such that $\eta[Gx_0]\subseteq \Inv(X)$.
	\end{enumerate}

	\item The following are equivalent:
	\begin{enumerate}
		\item The $G$-ambit $X$ has EDRP.
		\item There exists $\eta\in E(X)$ such that $\im(\eta)\subseteq \Inv(X)$.
		\item Every minimal subflow of $X$ is trivial.
	\end{enumerate}
\end{enumerate}
\end{corollary}

We state a few more conditions equivalent to EDRP.

\begin{remark}\label{remark: fixed points in powers}
For any flow $(G,X)$ the following conditions are equivalent:
\begin{enumerate}
\item $(G,X)$ has the property that any subflow of any finite Cartesian power
$X^n$ has a fixed point.

\item There is an element $\eta$ in $E(X)$ with $\im(\eta)\subseteq \Inv(X)$.
\end{enumerate}
\end{remark}

\begin{proof} 
(1) $\Rightarrow$ (2). By (1), for any finite tuple $\bar x$ of
elements of $X$ there is a fixed point $\bar y$ in the closure of the
orbit $G\bar x$. On the other hand, there is an element $\eta_{\bar x} \in
E(X)$ mapping $\bar x$ to $\bar y$. Then all the $\eta_{\bar x}$'s (for
varying $\bar x$) form a net, whose accumulation point $\eta \in E(X)$
has image contained in $\Inv(X)$ (because $\Inv(X)$ is closed). So (2)
holds.

(2) $\Rightarrow$ (1). Take $\eta \in E(X)$ satisfying (2). Since $\eta \in E(X)$, it
preserves each subflow $Y$ of $(G,X^n)$ as a set, so any point in
$\eta[Y]$ is in $Y$ and is also a fixed point (by the choice of $\eta$). So
(1) holds.
\end{proof}

\begin{remark}\label{remark: EDERP iff s-EDERP for all s} 
The ambit $(G,X,x_0)$ has EDRP iff for every $s\in X$ the ambit $(G,X,x_0)$ has $\{s\}$-DRP.
\end{remark}

\begin{proof}
The implication ($\Rightarrow$) is obvious. For the converse, first we show

\begin{clm}
For every finite $S \subseteq X$ the ambit $(G,X,x_0)$ has $S$-DRP.
\end{clm}

\begin{clmproof}
Write $S=\{s_0,\dots,s_{n-1}\}$, and construct recursively $\eta_0,\dots,\eta_{n-1} \in E(X)$ as follows. By the assumption applied to $s_0$ and Proposition \ref{prop: characteriazation_S-derp}, we can choose $\eta_0 \in E(X)$ with $\eta_0(s_0) \in \Inv(X)$. Again by assumption, but this time applied to $\eta_0(s_1)$, and Proposition \ref{prop: characteriazation_S-derp}, we can choose $\eta_1 \in E(X)$ with $\eta_1(\eta_0(s_1)) \in \Inv(X)$.  We continue with this procedure $n$ times. Then $\eta:=\eta_{n-1} \circ \dots \circ \eta_0 \in E(X)$ has the property that $\eta[S] \subseteq \Inv(X)$. Therefore, $(G,X,x_0)$ has $S$-DRP by Proposition \ref{prop: characteriazation_S-derp}.
\end{clmproof}

By Claim 1 and Proposition \ref{prop: characteriazation_S-derp}, for every finite $S \subseteq X$, there is $\eta_S \in E(X)$ with $\eta[S] \subseteq \Inv(X)$. Let $\eta$ be an accumulation point of the net $(\eta_S)_S$, where $S$ ranges over the finite subsets of $X$ ordered by inclusion. Then $\im(\eta) \subseteq \Inv(X)$, so $(G,X,x_0)$ has $S$-DRP by Proposition \ref{prop: characteriazation_S-derp}.
\end{proof}

Let us now briefly explain that the above notions WDRP and EDRP are strongly related (in fact, equivalent) to suitable notions of finite oscillation stability from \cite{The}. This is only for information, so we will not give full justifications. A quick (not informative) justification is that the dynamical characterizations in Corollary \ref{cor: amenability_derp} (1b) and (3c) are the same as those in in Propositions 2.3 and 2.4 in \cite{The}. But let us describe a more direct connection between the definitions.

Let $\mathcal{A}$ be the $C^*$-subalgebra of the algebra $C_b(G)$ of bounded complex-valued functions on $G$ which are of the form $h \circ p_{x_0}$ for some $h \in C(X)$, where $C(X)$ is the $C^*$-algebra of complex-valued functions on $X$ and $p_{x_0} \colon G \to X$ is given by $p_{x_0}(g):=gx_0$; note that since $Gx_0$ is dense in $X$, the assignment $h\circ p_{x_0} \mapsto h$ yields a well-defined isomorphism between $\mathcal{A}$ and $C(X)$. Let $G^{\mathcal{A}}$ be the Gelfand space of $\mathcal{A}$, i.e.\ the space of $C^*$-algebra homomorphisms from $\mathcal{A}$ to $\mathbb{C}$ with the weak $*$-topology, and $\tilde{e}_G \in G^\mathcal{A}$ be given by $\tilde{e}_G(f):=f(e_G)$. A classical fact based on Gelfand duality tells us that the map $\Phi \colon X \to G^{\mathcal{A}}$ given by $\Phi(x)(f):=h(x)$, where $h \in C(X)$ is unique such that $f=h \circ p_{x_0}$, is an isomorphism between the $G$-ambits $(G,X,x_0)$ and $(G,G^{\mathcal{A}},\tilde{e}_G)$.

Let us call a coloring $c \colon X \to 2^n$ {\em opp-strongly definable} if $c'$ given by $c'(g):=c(g^{-1})$ is strongly definable. By $0$-dimensionality of $X$ and Stone-Weierstrass theorem, the algebra $\mathcal{A}$ is the closure of the (unital) $*$-subalgebra of $C_b(G)$ generated by the opp-strongly definable colorings.

Having the above background, the property studied in \cite{The} (see Definition 2.1  and Proposition 2.3 in there) that each finite subfamily of $\mathcal{A}$ is finite oscillation stable is easily seen to be equivalent to WDRP. Similarly, using also Remark \ref{remark: EDERP iff s-EDERP for all s}, the property from \cite[Proposition 2.4]{The} that for every $x \in X$ and every finite family $\CF \subseteq C(X)$ the family $\{f \circ p_{x}: f \in \CF\}$ (where $p_x \colon G \to X$ is given by $p_x(g):=gx$) is finitely oscillation stable is easily seen to be equivalent to EDRP.


Now, we will investigate the relationships between WDRP, DRP, and EDRP. We clearly have:
$$\textrm{EDRP} \Rightarrow \textrm{DRP} \Rightarrow \textrm{WDRP}.$$

How about $\textrm{EDRP} \Leftarrow \textrm{DRP}$? In Subsection \ref{subsection: first order theories}, we will see that in the specialization of the general abstract context to first order theories, the properties EDRP and DRP are precisely the properties EDEERP and DEERP from \cite{KLM}, and so Example 5.7 from \cite{KLM} shows that $\textrm{EDRP} \nLeftarrow \textrm{DRP}$. A counterexample in the specialization to definable groups will appear later in Example \ref{example: DERP but not EDERP fr definable groups}. However, we also have a positive observation which in particular implies that in the specialization to the classical KPT context  we have $\textrm{EDRP} \Leftrightarrow \textrm{DRP} \Leftrightarrow \textrm{WDRP}$ (see Remark \ref{remark: all RP agree in Sigma^M}).

\begin{corollary}\label{corollary: equivalence of all RP for d-closed algebras}
Assume the $0$-dimensional ambit $(G,X,x_0)$ possess a left-continuous semigroup operation $*$ extending the action of $G$ (in the sense that $(gx_0)*x=gx$ for all $g \in G$ and $x \in X$). Equivalently,  $(G,X,x_0) \cong (G,S(\mathcal{A}), p_e)$ for some d-closed Boolean $G$-algebra of subsets of $G$. Then $\textrm{EDRP} \Leftrightarrow \textrm{DRP} \Leftrightarrow \textrm{WDRP}$.
\end{corollary}

\begin{proof}
The equivalence between the first two sentences follows from Facts \ref{fact: d-closed iff semigroup} and \ref{fact: 0-dim ambit is S(A)}. By the same facts (or by an easy direct argument), we get that the function $l\colon X \to E(X)$ given by $l(x)(y):=x*y$ is an isomorphism of $G$-flows.

Assume WDRP. By Corollary \ref{cor: amenability_derp}(1), there exists $x \in \Inv(X)$. Then, for any $g\in G$ and $y \in X$ we have $(gl(x))(y)= (l(gx))(y)=l(x)(y)=x*y$, and so $gl(x)=l(x)$. Thus, we get EDRP by Corollary \ref{cor: amenability_derp}(3).
\end{proof}

The question whether $\textrm{DRP} \Leftarrow \textrm{WDRP}$ was more difficult for us, because (as we will see) in all three specializations considered in Section \ref{section: specializations} this implication holds. To understand better the relationship between WDRP and DRP and to give a counterexample, we first make some general observations.

For a clopen subset $U$ of $X$, $y \in X$, and a finite subset $A$ of $G$, we will be considering the set: 
$$S(U,y,A):=\{h \in G: A^{-1}h^{-1}y \subseteq U \textrm{ or } A^{-1}h^{-1}y \subseteq U^c\}.$$
The following remark, which follows directly from the definition of $S$-DRP, is very useful.

\begin{remark}\label{remark: S-DERP in terms of S}
The ambit $(G,X,x_0)$ has $S$-DRP if and only if the family 
$$\{S(U,y,A): U \subseteq X \textrm{ clopen}, y \in S, A \subseteq G \textrm{ finite}\}$$
has the finite intersection property.
\end{remark}

When $\mathcal{A}$ is a Boolean $G$-algebra of subsets of $G$, $U \in \mathcal{A}$, $g \in G$, and $A \subseteq G$, we denote
$$S(U,g,A):=S([U],p_g,A)=\{h \in G: A^{-1}h^{-1}g \subseteq U \textrm{ or } A^{-1}h^{-1}g \subseteq U^c\},$$
where $[U]$ is the basic clopen in $S(\mathcal{A})$ given by $U$, $p_g \in S(\mathcal{A})$ is the principal ultrafilter at $g$, and $S([U],p_g,A)$ is computed with respect to the ambit $(G,S(\mathcal{A}),p_e)$.

The following corollary will imply that in all three specializations $\textrm{DRP} \Leftrightarrow \textrm{WDRP}$.

\begin{corollary}\label{corollary: right invariant G-algebras}
Let $\mathcal{A}$ be a Boolean (left) $G$-algebra of subsets of $G$ which is also a right $G$-algebra (i.e.\ closed under the right translations by the elements of $G$). Then for the ambit $(G,S(\mathcal{A}),p_e)$ the properties DRP and WDRP are equivalent.
\end{corollary}

\begin{proof}
By Remark \ref{remark: S-DERP in terms of S}: 
\begin{eqnarray*}
\textrm{DRP} & \text{iff} & \{S(U,g,A): U \in \mathcal{A}, g \in G, A \subseteq G \textrm{ finite}\} \textrm{ has the finite intersection property},\\
\textrm{WDRP} & \text{iff} & \{S(U,e,A): U \in \mathcal{A}, A \subseteq G \textrm{ finite}\} \textrm{ has the finite intersection property}.
\end{eqnarray*}
The right hand sides are equivalent, because,
by assumption, for $U \in \mathcal{A}$ and $g \in G$ we have $Ug^{-1} \in \mathcal{A}$ and we easily compute that $S(U,g,A)=S(Ug^{-1},e,A)$ for any finite $A \subseteq G$.
\end{proof}

Recall that a subset $U$ of a group $G$ is said to be {\em left [right] syndetic} (or {\em left [right] generic}) if finitely many left [right] translates of $U$ cover $G$.

\begin{lemma}\label{lemma: WDERP and not DERP}
Let $G$ be a group and $U \subseteq G$ a subset which is not left syndetic but for some finite $A \subseteq G$ the subset $AU \cap AU^c$ is right syndetic. Let $\mathcal{A}$ be the Boolean (left) $G$-algebra of subsets of $G$ generated by $U$. Then, the ambit $(G,S(\mathcal{A}),p_e)$ has WDRP but does not have DRP.
\end{lemma}

\begin{proof}
By Remark \ref{remark: S-DERP in terms of S}, $(G,S(\mathcal{A}),p_e)$ has WDRP if and only if for every $n \geq 1$, for every $U_1,\dots,U_n \in \mathcal{A}$, and for every finite $A_1,\dots,A_n \subseteq G$, we have 
$$(A_1U_1 \cap A_1U_1^c) \cup \dots \cup (A_nU_n \cap A_nU_n^c) \ne G.$$

Since each $U_i$ is a Boolean combination of left translates of $U$, there is a finite $B \subseteq G$ such that for every $i$, either $U_i \subseteq BU$  or $U_i^c \subseteq BU$. Thus, there is a finite $B' \subseteq G$ such that for every $i$, $A_iU_i \cap A_iU_i^c \subseteq B'U$. Hence, since $U$ is not left syndetic, we conclude that $(A_1U_1 \cap A_1U_1^c) \cup \dots \cup (A_nU_n \cap A_nU_n^c) \ne G$, and so $(G,S(\mathcal{A}),p_e)$ has WDRP.

On the other hand, since $AU \cap AU^c$ is right syndetic, there are $g_1,\dots,g_n \in G$ such that 
$$(AU \cap AU^c)g_1^{-1} \cup \dots \cup (AU \cap AU^c)g_n^{-1} =G.$$ 
But this precisely means that 
$$S(U,g_1,A) \cap \dots \cap S(U,g_n,A) = \emptyset,$$
which by Remark \ref{remark: S-DERP in terms of S} implies that $(G,S(\mathcal{A}),p_e)$ does not have DRP.
\end{proof}

So in order to find a counterexample to the implication $\textrm{DRP} \Leftarrow \textrm{WDRP}$, it remains to find an example satisfying the assumptions of Lemma \ref{lemma: WDERP and not DERP}. It is easy to see that if there is such an example, then there is one for $G$ being a free group. We will find it in the free group $F_2$.

\begin{example}\label{example: WDERP does not imply DERP}
Let $G:=F_{a,b}$ be the free group in free generators $a,b$. Define $U$ to be the set of all $x \in F_2$ such that the reduced word of $x$ ends in $a$ and the parity of the total exponent of $a$ in this word is $0$ (i.e.\ the total number of positive and negative occurrences of $a$ in this word is even). We claim that the assumptions of Lemma \ref{lemma: WDERP and not DERP} are satisfied for $A:=\{e,a\}$. 
\end{example}

\begin{proof}
Let us first check that $U$ is not left syndetic. Consider any finite $C \subseteq G$. Let $N \in \mathbb{N}$ be greater than the lengths of the reduced words of all elements in $C$. Then for every $c \in C$ we have that $c^{-1}b^N \notin U$ (as the reduced word of $c^{-1}b^N$ ends in $b$), so $b^N \notin cU$. Thus, $b^N \notin CU$, hence $CU \ne G$.

Now, we show that $AU \cap AU^c$ is right syndetic. First, notice that $AU \cap AU^c \supseteq U \cap aU^c \supseteq U$. Indeed, the inclusion $U \subseteq aU^c$ holds, because for any $u \in U$ the total number of occurrences of $a^{\pm 1}$ in the reduced word of $a^{-1}u$ is odd, and so $a^{-1}u \in U^c$, i.e.\ $u \in aU^c$. 
Thus, it remains to show that $U$ is right syndetic. We will check that $UB=G$ for $B:=\{a^{-1},a^{-2}, (ba)^{-1},(ba^2)^{-1}\}$. 

Consider any $g \in G$. If the reduced word of $g$ does not end in $a^{-1}$, then the reduced words of $ga$ and $ga^2$ both end in $a$ and one of them has the parity of the total exponent of $a$ equal to $0$, so $ga \in U$ or $ga^2 \in U$,  hence $g \in UB$. If the reduced word of $g$ ends in $a^{-1}$, then the reduced words of $gba$ and $gba^2$ end in $a$ and one of them has the parity of the total exponent of $a$ equal to $0$, so $gba \in U$ or $gba^2 \in U$, hence $g \in UB$.
\end{proof}

The whole families of examples of $0$-dimensional ambits with DRP or EDRP will be provided by specializations to first order theories, definable groups, and classical KPT context in Section \ref{section: specializations}.

\subsection{Externally definable Ramsey degrees}\label{subsection: 3.2}

\subsubsection{Externally definable Ramsey degrees and profiniteness of the Ellis group}\label{subsubsection: Ramsey and profiniteness}

The goal of this subsection is to define variants of definable Ramsey degrees and use them to give a Ramsey-theoretic characterization of profiniteness of the Ellis group of the $0$-dimensional ambit $(G,X,x_0)$. We also show that the last property implies that all distal minimal factors of $(G,X)$ are profinite $G$-flows (i.e.\ inverse limits of finite $G$-flows).

Recall that $\CB$ denotes the Boolean $G$-algebra of all clopen subsets of $X$. We write $G\backslash \CB$ for the set of all $G$-orbits $GU$, where $U$ ranges over $\CB$.

Let $S(X)$ be the Stone space $S(\CB)$. For $\Delta\subseteq G\backslash \CB$, let $\CA(\Delta)$ be the Boolean $G$-algebra generated by $\bigcup \Delta$ and let $S_{\Delta}(X):=S(\CA(\Delta))$ be the Stone space of $\CA(\Delta)$. Then  $(G,S_{\Delta}(X),p(x_0)_{\Delta})$ is a $G$-ambit under the natural left action of $G$, where $p(x_0)_{\Delta}:=\{V\in \mathcal{A}(\Delta): x_0\in V\}$.

Let $\CP_{fin}(G\backslash \CB)$ be the set of finite subsets of $G\backslash \CB$, which is directed by inclusion. For each $\Delta_0\subseteq \Delta_1$ from $\CP_{fin}(G\backslash \CB)$, define a map $\pi^{\Delta_1}_{\Delta_0}\colon S_{\Delta_1}(X)\rightarrow S_{\Delta_0}(X)$ by $\pi^{\Delta_1}_{\Delta_0}(p):=p_{\Delta_0}$, where $p_{\Delta_0}:=p\cap \CA(\Delta_0)=\{ U \in \mathcal{A}(\Delta_0): U \in p\}$. Then, each $\pi^{\Delta_1}_{\Delta_0}$ is an epimorphism of $G$-flows so that $(S_{\Delta}(X),\pi^{\Delta_1}_{\Delta_0})_{\Delta_0\subseteq \Delta_1}$ forms an inverse system. The following important remark is obtained by a routine argument.

\begin{remark}\phantomsection\label{rem: identifying_S(X)}
\begin{enumerate}
	\item $S(X)\cong \varprojlim_{\Delta\in \CP_{fin}(G\backslash \CB)}\limits S_{\Delta}(X)$ via  $p \mapsto (p_{\Delta})_\Delta$, where $p_\Delta:= p\cap \CA(\Delta)$.
	\item $X\cong S(X)$ via $x\mapsto p(x)$, where $p(x):=\{U\in \CB : x\in U\}$.
\end{enumerate}
\end{remark}

For each $\Delta \in \CP_{fin}(G\backslash \CB)$ we also have the $G$-flow epimorphism $\pi_\Delta \colon X \to S_{\Delta}(X)$ given by $\pi_\Delta(x): = p(x)_\Delta:= p(x) \cap \CA(\Delta)=\{U \in \CA(\Delta): x \in U\}$.  For an arbitrary $S \subseteq X$ define $S_{\Delta}(S):=\pi_\Delta[S]$.

\begin{definition}\label{definition: Sigma-colorings}
Let $S\subseteq X$ and $\Sigma\subseteq \CB$.
A coloring $c\colon G\rightarrow 2^n$ is called an {\em $S$-definable $\Sigma$-coloring} if $c$ is an $S$-definable coloring given by $U_0,\ldots,U_{n-1}\in \Sigma$ (see Definition \ref{definition: definable colorings}). When $S=Gx_0$, we call it a {\em definable $\Sigma$-coloring}; when $S=X$, we call it an {\em externally definable $\Sigma$-coloring}.
\end{definition}	

\begin{definition}\phantomsection\label{definition: sep. fin EDERdeg}
\begin{enumerate}
\item Let $S$ be a subset of $X$. The $G$-ambit $X$ has {\em separately finite $S$-definable  Ramsey degree} (sep.\ fin.\ $S$-DRdeg) if for any finite $\Sigma\subseteq \CB$, there exists $l(=l(S,\Sigma))\ge 1$ such that for every $S$-definable $\Sigma$-coloring $c\colon G\rightarrow 2^n$, and for every $k\ge 1$ and $A\in [G]^k$, there exists $h\in G$ such that $|c[hA]|\le l$.
\item The $G$-ambit $X$ has  {\em separately finite definable  Ramsey degree} (sep.\ fin.\ DRdeg) if it has sep.\ fin.\ $Gx_0$-DRdeg; it has {\em separately finite externally definable  Ramsey degree} (sep.\ fin.\ EDRdeg) if it has sep.\ fin.\ $X$-DRdeg.
\end{enumerate}
\end{definition}

By analogy with Definition \ref{definition: S-DERP}, we could also define {\em separately finite weakly definable  Ramsey degree} as sep.\ fin.\ $\{x_0\}$-DRdeg, but this property trivially holds in every 0-dimensional ambit, so it does make sense to consider it.

Here is the main abstract result which generalizes Theorem 3.20 in \cite{KLM}. The proof uses the idea of the proof of that theorem, but in fact it is more transparent and simpler (by avoiding formulas, variables, etc.). The most important case of that theorem, which we will be using later, is stated as Corollary  \ref{corollary: characterization_sep_fin_EDERdeg} below.

\begin{theorem}\label{theorem: characterization_sep_fin_S-DERdeg}
Let $S \subseteq X$. The following conditions are equivalent:
\begin{enumerate}
	\item The $G$-ambit $X$ has sep.\ fin.\ $S$-DRdeg.
	\item For every finite subset $\Delta$ of $G\backslash \CB$ there exists $\eta\in E(S_{\Delta}(X))$ such that $\eta[S_{\Delta}(S)]$ is finite.
\end{enumerate}
\end{theorem}

\begin{proof}
$(\Rightarrow)$ Consider $\Delta:=\{GU_0,\ldots,GU_{k-1}\}\subseteq G\backslash \CB$ finite. For $\Sigma:=\{U_0,\ldots,U_{k-1}\}\subseteq \CB$, let $l=l(S,\Sigma)$ be a witness for sep.\ fin.\ $S$-DRdeg with respect to $\Sigma$. For each $\bar y=(y_0,\ldots,y_{n-1})\in S^n$, let $c_{\bar y}:G\rightarrow 2^{kn}$ be an $S$-definable $\Sigma$-coloring defined as follows: for $g\in G$, $i <k$, $j<n$,  $$c_{\bar y}(g)(i,j):=
\begin{cases}
1 & \mbox{ if } y_j\in gU_i\\
0 & \mbox{ if } y_j\notin gU_i
\end{cases}.$$
By the choice of $l$, for each $A\in [G]^s$ there exists $g_{\bar y,A}\in G$ such that $|c_{\bar y}[g_{\bar y,A}A]|\le l$. Direct the set $\bigcup_{n,s} S^n \times [G]^s$ by: 
$$(\bar{y},A) \leq (\bar{y}',A')$$ 
iff $A \subseteq A'$, $|\bar y| \leq |\bar y'|$, and $\{y_i\}_{i} \subseteq \{y'_j\}_j$.
Treating elements of $G$ as elements of $E(S_{\Delta}(X))$, we may choose $\eta \in  E(S_{\Delta}(X))$ to be  the limit of a convergent subnet of the net $(g_{\bar y, A}^{-1})_{\bar{y}, A}$.

\begin{clm}
$\eta[S_{\Delta}(S)]$ is finite.
\end{clm}

\begin{clmproof}
Suppose $\eta[S_{\Delta}(S)]$ is infinite. There exist $z_0,z_1,\ldots\in S$ such that the $\eta(p(z_t)_{\Delta})$'s are pairwise distinct. For $t\neq t'$, $\eta(p(z_t)_{\Delta})\neq \eta(p(z_{t'})_{\Delta})$ is witnessed by some $g_{t,t'}U_i\in GU_i$ in the sense that
$$g_{t,t'}U_{i}\in\eta(p(z_t)_{\Delta})\Leftrightarrow g_{t,t'}U_{i}^c\in\eta(p(z_{t'})_{\Delta}).$$ 
By Ramsey theorem, we may assume that there exists $i_0<k$ such that for every $t<t'$ the inequality $\eta(p(z_t)_{\Delta})\neq\eta(p(z_{t'})_{\Delta})$ is witnessed by some $g_{t,t'}U_{i_0}\in GU_{i_0}$, i.e.\ 
$$g_{t,t'}U_{i_0}\in\eta(p(z_t)_{\Delta})\Leftrightarrow g_{t,t'}U_{i_0}^c\in\eta(p(z_{t'})_{\Delta}).$$ 

Take $n>2^l$. Then, we have that 
$$\bigwedge_{t<t'<n}\left( g_{t,t'}U_{i_0}\in\eta(p(z_t)_{\Delta})\Leftrightarrow g_{t,t'}U_{i_0}^c\in\eta(p(z_{t'})_{\Delta})\right),$$ 
which is an open condition on $\eta$. So there exists $(\bar y,A) \geq ((z_0,\ldots,z_{n-1}),\{g_{t,t'}:t<t'<n\})$ such that 
$$\bigwedge_{t<t'<n}\left( g_{t,t'}U_{i_0}\in g_{\bar y,A}^{-1}(p(z_t)_{\Delta})\Leftrightarrow g_{t,t'}U_{i_0}^c\in g_{\bar y,A}^{-1}(p(z_{t'})_{\Delta})\right),$$ 
which is equivalent to 
\begin{equation}\tag{$*$}
\bigwedge_{t<t'<n}\left( g_{\bar y,A}g_{t,t'}U_{i_0}\in p(z_t)_{\Delta}\Leftrightarrow g_{\bar y,A}g_{t,t'}U_{i_0}^c\in p(z_{t'})_{\Delta}\right).
\end{equation}



Put $h:=g_{\bar y,A}$; recall, $|c_{\bar y}[hA]|\le l$. Since $(\bar y,A) \geq ((z_0,\ldots,z_{n-1}),\{g_{t,t'}:t<t'<n\})$, $hg_{t,t'}\in hA$ for all $t<t'<n$, and for each $t<n$ we can find $i_t$ such that $z_t=y_{i_t}$.

For each $g\in hA$, put $S(g):=\{t<n: y_{i_t}\in gU_{i_0}\}$. Note that the condition $y_{i_t}\in gU_{i_0}$ is equivalent to $gU_{i_0}\in p(y_{i_t})_{\Delta}$. Since $|c_{\bar y}[hA]|\le l$, 
$$l':=|\{S(g):g\in hA\}|\le l,$$ 
and say $\{S(g):g\in hA\}=\{S_0,\ldots,S_{l'-1}\}$. Consider the function $$f\colon n\rightarrow \CP(\{S_0,\ldots,S_{l'-1}\}),\; t\mapsto\{S_u:t\in S_u\},$$ which is injective. Indeed, using $(*)$ and $z_t=y_{i_t}$, for each $t<t'<n$ we have:
\begin{eqnarray*}
&&\big(y_{i_t}\in hg_{t,t'}U_{i_0}\Leftrightarrow y_{i_{t'}}\in hg_{t,t'} U_{i_0}^c\big)\\
&\Rightarrow &\big( t\in S(hg_{t,t'}) \Leftrightarrow t'\notin S(hg_{t,t'})\big)\\
&\Rightarrow & \big(S(hg_{t,t'})\in f(t) \Leftrightarrow S(hg_{t,t'})\notin f(t')\big)\\
&\Rightarrow & \big(f(t)\neq f(t')\big).
\end{eqnarray*}
So, we have $n\le 2^{l'}\le 2^l<n$, a contradiction.
\end{clmproof}

$(\Leftarrow)$ Take $\Sigma\subseteq \CB$ finite. Put $\Delta:=\{GU:U\in \Sigma\}$, which is finite. By assumption, there exists $\eta\in E(S_{\Delta}(X))$ such that $\eta[S_{\Delta}(S)]$ is finite. Put $k:=|\Sigma|$ and $t:=|\eta[S_{\Delta}(S)]|$. Let $c\colon G\rightarrow 2^n$ be an $S$-definable $\Sigma$-coloring given by $U_0,\ldots,U_{n-1}\in \Sigma$ and $y_0,\ldots,y_{n-1}\in S$. Note that $c$ depends only on $p(y_i)_{\Delta}$ because 
$$y_i\in gU_i\Leftrightarrow gU_i\in p(y_i)\Leftrightarrow gU_i\in p(y_i)_{\Delta}.$$

\begin{clm}
For all $A\in [G]^s$, there exists $h\in G$ such that $|c[hA]|\le 2^{kt}$. 
\end{clm}

\begin{clmproof}
Define $\epsilon \colon G \rightarrow 2^n$ via 
$$gU_i^{\epsilon(g)(i)}\in \eta(p(y_i)_{\Delta}),$$
where for a clopen subset $U$ of $X$, $U^1:=U$ and $U^0:=U^c$. Note that there are at most $2^{kt}$-many choices for $\epsilon(g)$, that is, $|\epsilon[G]|\le 2^{kt}$. For an arbitrary $A:=\{g_0,\ldots,g_{s-1}\}\in [G]^s$, consider an open condition on $\eta$ given by
$$\bigwedge_{j<s}\bigwedge_{i<n}\big( g_jU_i^{\epsilon(g_j)(i)}\in \eta(p(y_i)_{\Delta})\big).$$ Then, for some $h\in G$:
\begin{eqnarray*}
&&\bigwedge_{j<s}\bigwedge_{i<n}\big( g_jU_i^{\epsilon(g_j)(i)}\in h^{-1}(p(y_i)_{\Delta})\big)\\
&\Rightarrow & \bigwedge_{j<s}\bigwedge_{i<n}\big( hg_jU_i^{\epsilon(g_j)(i)}\in p(y_i)_{\Delta}\big)\\
&\Rightarrow & \bigwedge_{j<s}\bigwedge_{i<n} hg_jU_i^{\epsilon(g_j)(i)}\ni y_i,
\end{eqnarray*}
and we have that $$|c[hA]|\le |\epsilon[A]|\le |\epsilon[G]|\le 2^{kt}.$$
\end{clmproof}
This finishes the proof of the theorem.
\end{proof}

\begin{corollary}\label{corollary: characterization_sep_fin_EDERdeg}
The following conditions are equivalent:
\begin{enumerate}
	\item The $G$-ambit $X$ has sep.\ fin.\ EDRdeg.
	\item For every finite subset $\Delta$ of $G\backslash \CB$ there exists $\eta\in E(S_{\Delta}(X))$ with finite image.
\end{enumerate}
\end{corollary}


\begin{corollary}\label{corollary: Ellis group is profinite}
If the $G$-ambit $X$ has sep.\ fin.\ EDRdeg, then the Ellis group of $X$ is profinite.
\end{corollary}

\begin{proof}
By Corollary \ref{corollary: characterization_sep_fin_EDERdeg}, for each finite subset $\Delta$ of $G\backslash \CB$ there exists $\eta\in E(S_{\Delta}(X))$ with finite image, and so the Ellis group of the flow $(G,S_{\Delta}(X))$ is finite by \cite[Fact 2.1]{KLM}. Hence, by Remark \ref{rem: identifying_S(X)} and \cite[Fact 2.8]{KLM}, the Ellis group of the flow $(G,X)$ is profinite.
\end{proof}

The next structural corollary requires the following structural fact on distal minimal flows, which we deduce from some classical results in topological dynamics.

\begin{fact}\label{fact: prof. Ellis group gives prof. flow}
If a distal minimal flow $(G,X)$ has a profinite Ellis group, then it is a profinite flow (i.e.\ an inverse limit of finite flows).
\end{fact}

\begin{proof}
By distality and \cite[Theorem I.3.3(2)]{Gla}, $E(X)$ is a group, so $E(X)$ is the unique minimal left ideal in $E(X)$ and the unique Ellis group of $X$. Since by assumption this Ellis group is profinite in the $\tau$-topology and so compact (Hausdorff), we conclude that the $\tau$-topology coincides with the usual topology on $E(X)$. Thus, $E(X)$ is a compact topological group. By minimality of $X$, the group $E(X)$ acts on $X$ transitively, so, using compactness and limits of nets, we conclude that $E(X)$ consists of homeomorphisms. By virtue of \cite[Theorem I.3.3(4,5)]{Gla}, we obtain that $(G,X)$ is isomorphic to $(G,E(X)/H)$ for some closed subgroup $H$ of $E(X)$ (where $g(\eta/H):=(g\eta)/H$). As $E(X)$ is profinite, it has a basis $\{N_i\}_{i \in I}$ of open neighborhoods of $\id$ which consists of clopen normal subgroups $N_i$. It is now clear that the flow $(G,E(X)/H)$ is isomorphic to $(G,\invlim E(X)/(N_iH))$ which is a profinite flow.
\end{proof}

\begin{corollary}\label{corollary: sep. fin. EDERdeg implies profiniteness of distal factors}
If the $G$-ambit $X$ has sep.\ fin.\ EDRdeg, then every 
distal minimal factor of $(G,X)$ is a profinite flow\footnote{Note that profiniteness of the flow $(G,X)$ is a much stronger structural information than our usual assumption that the space $X$ is $0$-dimensional (equivalently, profinite).}.
\end{corollary}

\begin{proof}
It follows from Corollary \ref{corollary: Ellis group is profinite} and Fact \ref{fact: prof. Ellis group gives prof. flow}.
\end{proof}

In a forthcoming joined paper of the first author with Daniel Hoffmann, Bohr compactifications below ambits will be studied. In the following short discussion on those compactifications, one can drop the assumption that $(G,X,x_0)$ is $0$-dimensional.

\begin{definition}\label{definition: Bohr compactification below ambit}
Let $(G,X,x_0)$ be an ambit. By a {\em compactification of $G$ below the ambit $(G,X,x_0)$} we mean a group homomorphism $h \colon G \to K$ with dense image for some compact group $K$ such that $h = f_1 \circ f_2$, where $f_2 \colon G 
\to X$ is given by $f_2(g):=gx_0$ and $f_1 \colon X \to K$ is an epimorphism of ambits (mapping $x_0$ to $e_K$) where the $G$-action on $K$ is given by $g\cdot k:=h(g)k$. The {\em Bohr compactification of $G$ below $(G,X,x_0)$} is a universal compactification $\mathfrak{b}$ of $G$ below $(G,X,x_0)$. Abusing terminology, the target group $K$ of $\mathfrak{b}$ is also called the Bohr compactification of $G$ below $(G,X,x_0)$.
\end{definition}

\begin{remark}
For any ambit $(G,X,x_0)$ a Bohr compactification of $G$ below $(G,X,x_0)$ exists and is unique up to isomorphism.
\end{remark}

\begin{proof}
Uniqueness follows from the density of the image. For the existence, take a set $\{\mathfrak{b}_i\colon G \to K_i\}_{i \in I}$ of all (up to isomorphism) compactifications of $G$ below $(G,X,x_0)$. Put $K:=\cl(\{(\mathfrak{b}_i(g))_{i \in I}): g \in G\}) \subseteq \prod_{i \in I} K_i$. It is clear that the map $\mathfrak{b} \colon G \to K$ given by $\mathfrak{b}(g):=(\mathfrak{b}_i(g))_{i \in I}$ is a compactification of $G$ below $(G,X,x_0)$ which is universal.
\end{proof}

Recall that whenever $(G,X)$ is a flow, then $(G,E(X), \id)$ is an ambit. Moreover, if $(G,X,x_0)$ is an ambit, then there is an epimorphism of ambits $(G,E(X),\id) \to (G,X,x_0)$ given by $\eta \mapsto \eta(x_0)$. Thus, it follows from the definitions that for any ambit $(G,X,x_0)$ there is a unique continuous epimorphism from the Bohr compactification of $G$ below  $(G,E(X), \id)$ to the one below $(G,X,x_0)$.

In the aforementioned forthcoming paper, the following fact will be proven.

\begin{fact}\label{fact: Bohr from a future paper}
Let $(G,X,x_0)$ be any ambit. Then there exists a group epimorphism and topological quotient map from the Ellis group $u\mathcal{M}$ of $X$ to the Bohr compactification of $G$ below $(G,E(X),\id)$ which factors through the largest Hausdorff quotient $u\mathcal{M}/H(u\mathcal{M})$. Thus, we obtain the induced epimorphism and topological quotient map from $u\mathcal{M}/H(u\mathcal{M})$ to the Bohr compactification of $G$ below $(G,E(X),\id)$.
\end{fact}

Now, we return to the context of the $0$-dimensinal ambit $(G,X,x_0)$. Let $\mathcal{M}$ be a minimal left ideal in $E(X)$ and $u$ an idempotent in $\mathcal{M}$.
The proof of Theorem 3 from \cite{KLM} easily adapts to the abstract context, yielding the next proposition; except the implications involving Bohr compactifications which require Fact \ref{fact: Bohr from a future paper}. For the reader's convenience we give a proof avoiding the model theory context used in \cite[Theorem 3]{KLM}.

\begin{proposition}\label{proposition: relationships}
Consider the following conditions.
\begin{enumerate}[leftmargin=3.7em,align=left,labelwidth=2.2em]
\item[($A'''$)] The Bohr compactification of $G$ below $(G,X,x_0)$ is profinite.
\item[($A''$)] The Bohr compactification of $G$ below $(G,E(X),\id)$ is profinite.
\item[($A'$)] $u\mathcal{M}/H(u\mathcal{M})$ is profinite.
\item[($A$)] $u\mathcal{M}$ is profinite.
\item[($B$)] $X \cong  \invlim X_i$, where the Ellis groups of all the $X_i$'s are finite.
\item[($C$)] For every finite $\Delta \subseteq G \backslash \mathcal{B}$ the Ellis group of $S_\Delta(X)$ is finite.
\item[($D$)] For every finite  $\Delta \subseteq G \backslash \mathcal{B}$ there exists $\eta \in E(S_{\Delta}(X))$ with finite image.
\end{enumerate}
Then $(D) \Rightarrow (C) \Leftrightarrow (B) \Rightarrow (A) \Rightarrow (A') \Rightarrow (A'') \Rightarrow (A''')$.
\end{proposition}

\begin{proof}
The implication  $(D)\Rightarrow(C)$  follows by \cite[Fact 2.1]{KLM};  $(C)\Rightarrow(B)$  by Remark \ref{rem: identifying_S(X)}(1);   $(B)\Rightarrow(A)$  by \cite[Fact 2.8]{KLM},   $(A)\Rightarrow(A')$  is trivial;   $(A')\Rightarrow(A'')$  by Fact \ref{fact: Bohr from a future paper};  $(A'')\Rightarrow(A''')$  by the existence of a continuous epimorphism between the two Bohr compactifications.

It remains to prove  $(B)\Rightarrow(C)$  and this is a straightforward adaptation of the proof of the same implication in the context of theories in \cite[Theorem 3]{KLM}.

By assumption and Remark \ref{rem: identifying_S(X)}, we have two presentations of $X$ as inverse limits of flows: $\invlim_{i \in I} X_i$ and $\varprojlim_{\Delta\in \CP_{fin}(G\backslash \CB)}\limits S_{\Delta}(X)$. 
For any $\Delta_0 \in \CP_{fin}(G\backslash \CB)$ and $i_0 \in I$, let $f_{\Delta_0} \colon \varprojlim_{\Delta\in \CP_{fin}(G\backslash \CB)}\limits S_{\Delta}(X) \to S_{\Delta_0}(X)$ and $g_{i_0} \colon \invlim_{i \in I} X_i \to X_{i_0}$ be the projections. 
For $\Delta \in \CP_{fin}(G\backslash \CB)$ and $i \in I$, define $F_\Delta$ and $G_i$ to be the fiber equivalence relations on $X$ induced by $f_\Delta$ and $g_i$, respectively. They are closed, $G$-invariant, and the families $\{F_\Delta\}_\Delta$ and $\{G_i\}_i$ are directed in the sense that $\Delta \subseteq \Delta'$ implies $F_{\Delta'} \subseteq F_{\Delta}$ and $i \leq i'$ implies $G_{i'} \subseteq G_i$.

Notice that $\bigcap_{\Delta \in \CP_{fin}(G\backslash \CB)} F_\Delta = \bigcap_{i \in I} G_i$ is the diagonal $D$ in $X \times X$.

\begin{clm}
For every $\Delta \in \CP_{fin}(G\backslash \CB)$ there exists $i \in I$ such that $G_i \subseteq F_\Delta$.
\end{clm}

\begin{clmproof}
Observe that for any clopen $C\subseteq X$ we can find $i_0\in I$ such that $C$ is a union of $G_{i_0}$-classes. To see this, note that $\bigcap_{i\in I}G_i=D$ implies that $\bigcup_{i\in I}G_i^c\cup (C\times C)\cup (C^c\times C^c)=X^2$, where each member of the union is open. By compactness, $\bigcup_{i\in I_0}G_i^c\cup (C\times C)\cup (C^c\times C^c)=X^2$ for some finite $I_0$, so $\bigcap_{i\in I_0}G_i\subseteq (C\times C)\cup (C^c\times C^c)$. By choosing $i_0$ to be greater than all elements of $I_0$, we get $G_{i_0}\subseteq (C\times C)\cup (C^c\times C^c)$, and the conclusion follows. 

Consider any $\Delta \in \CP_{fin}(G\backslash \CB)$, so $\Delta =\{GU_0,\dots,GU_{n-1}\}$ for some clopens $U_0,\dots,U_{n-1}$ in $X$. Then $f_\Delta^{-1}[[U_k]]$ is a clopen subset of $X$ (where $[U_k]$ is the basic clopen in $S_{\Delta}(X)$ given by $U_k$). By the previous paragraph, there exists $i_{k}  \in I$ such that $f_\Delta^{-1}[[U_k]]$ is a union of $G_{i_k}$-classes. Since $G_{i_k}$ is $G$-invariant and $f_\Delta$ is a homomorphism of $G$-flows, we get that for every $g \in G$ the preimage $f_\Delta^{-1}[[gU_k]]$ is also a union of $G_{i_k}$-classes. Choose $i \in I$ greater than all $i_k$'s for $k<n$. It follows that for every clopen $C$ in $S_\Delta(X)$ the preimage $f_\Delta^{-1}[C]$  is a union of $G_{i}$-classes. Therefore, for any $p \in S_\Delta(X)$, since 
$$f_\Delta^{-1}[\{p\}] =  \bigcap_C f_\Delta^{-1}[C]$$ 
with the intersection taken over all clopens $C$ containing $p$,
we get that $f_\Delta^{-1}[\{p\}]$ is a union of $G_i$-classes. This means that any $F_\Delta$-class is a union of $G_i$-classes, hence $G_i\subseteq F_\Delta$.
\end{clmproof}

By the claim, for each $\Delta\in\mathcal F$ we can find $i\in I$ such that the identity map $X \to X$ induces a flow epimorphism $X/G_i\to X/F_\Delta$. 
On the other hand, $X/G_i \cong X_i$ and $X/F_\Delta \cong S_\Delta(X)$ as flows. Therefore, there exists a flow epimorphism $X_i \to S_\Delta(X)$, and we are done by Facts 2.5(i) and 2.3 from \cite{KLM}.
\end{proof}

Note that, by Corollary \ref{corollary: characterization_sep_fin_EDERdeg}, the condition ($D$) above is precisely the dynamical characterization of the condition that $(G,X,x_0)$ has sep.\ fin.\ EDRdeg.
In the aforementioned forthcoming paper of the first author with Daniel Hoffmann, it will be shown that in the context of first order theories the Bohr compactification of $\Aut(\C)$ below the ambit $(\Aut(\C),S_{\bar c}(\C),\tp(\bar c/\C))$ is precisely $\Gal_{KP}(T)$, so in this context the above ($A'''$) becomes ($A''$) from \cite[Theorem 3]{KLM} (for the notation $\C$ and $S_{\bar{c}}(\C)$ see the second paragraph of Subsection \ref{subsection: first order theories}; for the definition of the Kim-Pillay Galois group $\Gal_{KP}(T)$ and related issues see \cite[Subsection 2.1]{KLM} and references in there). Profiniteness of $\Gal_{KP}(T)$ is very desirable, as it is equivalent to the equality of Shelah and Kim-Pillay strong types.  In \cite{KLM}, working in the context of theories, counterexamples for $(A''')\Rightarrow (A')$ and for $(A') \Rightarrow (B)$ were found. In fact, one can show that the latter counterexample really shows that $(A')$ does not imply $(A)$, but we will not discuss it here. The questions whether $(A) \Rightarrow (B)$ and whether $(C) \Rightarrow (D)$ remain open for theories. In Section \ref{section: questions and examples}, we will give a counterexample to $(A) \Rightarrow (B)$ in the general context of $0$-dimensional ambits. But we do not know if $(C) \Rightarrow (D)$ holds in general. In the context of a group $G$ definable in a structure $M$, by \cite[Proposition 3.4]{GPP}, the quotient map $G \to \bar{G}/\bar{G}^{00}_M$ is the definable Bohr compactification of $G$ (which is exactly the Bohr compactification below the ambit $(G,S_G(M),\tp(e/M))$), so item $(A''')$ is precisely the equality of the components $\bar{G}^{00}_M=\bar{G}^{0}_M$ (see Subsection \ref{subsection: definable groups} for the definitions and notation).

\subsubsection{Externally definable Ramsey degrees and metrizability of minimal left ideals in the Ellis semigroup}\label{subsection: joint. fin. EDERdeg}

The main goal of this subsection is to give a Ramsey-theoretic criterion of metrizability of the minimal left ideals in $E(X)$. It will be explained in Subsection \ref{subsection: KPT context} that the specialization of the obtained criterion to the classical KPT context recovers Zuecker's criterion of metrizbaility of the universal minimal flow of the group of automorphisms of a Fra\"{i}ss\'{e} structure from \cite[Theorem 8.7]{Zuc}.

We start from a Ramsey-theoretic characterization of finiteness of the minimal left ideals in $E(X)$.

\begin{definition}\label{definition: joint. fin. EDERdeg}
The $G$-ambit $X$ has {\em jointly finite externally definable  Ramsey degree} (joint.\ fin.\ EDRdeg) if there exists $l \geq 1$ such that for every 
externally definable coloring $c \colon G \to 2^n$, and for every $k \geq 1$ and $A \in [G]^k$, there exists $h \in G$ such that $|c[hA]| \leq l$.
\end{definition}

\begin{proposition}\label{proposition: joint fin EDERdeg}
The following conditions are equivalent.
\begin{enumerate}
\item $(G,X,x_0)$ has joint.\ fin.\ EDRdeg.
\item There exists $\eta \in E(X)$ whose orbit $G \eta$ is finite.
\item Some minimal left ideal in $E(X)$ is finite.
\item Every minimal left ideal in $E(X)$ is finite.
\end{enumerate}
\end{proposition}

\begin{proof}
The equivalence (3) $\Leftrightarrow$ (4) follows from the general fact that any two minimal left ideals in $E(X)$ are isomorphic as $G$-flows (see \cite[Proposition I.2.5]{Gla} for the proof for $\beta G$; the proof for $E(X)$ is the same).

(2) $\Rightarrow$ (3). Since $G\eta$ is finite, the left ideal $E(X) \eta=\cl(G \eta)$ equals $G\eta$. So a minimal left ideal contained in $E(X) \eta$ is finite.

(3) $\Rightarrow$ (2). Any element in a minimal left ideal has a finite $G$-orbit.

(1) $\Rightarrow$ (2) 
Choose $l$ witnessing joint.\ fin.\ EDRdeg, that is
for every $\bar{y} =(y_0,\dots,y_{n-1}) \in X^n$, $\bar{U}=(U_0,\dots,U_{n-1}) \in \CB^n$, and finite $A \subseteq G$ there exists $g_{\bar{y},\bar{U},A} \in G$ such that $|c[g_{\bar{y},\bar{U},A}A]|\leq l$, where for $g \in G$ the value $c(g) \in 2^n$ is given by: $c(g)(i) =1$ if $y_i \in gU_i$, and  $c(g)(i) =0$ if $y_i \in gU_i^c$ (note that $c=c_{\bar y,\bar U}$ depends on $\bar y$ and $\bar U$). Let $\eta$ be the limit of a convergent subnet of the net $(g_{\bar{y},\bar{U},A}^{-1})_{\bar{y},\bar{U},A}$ (the direct pre-order on the indices is defined as in the proof of Proposition \ref{prop: characteriazation_S-derp}).

We will show that $G\eta$ is finite, even $|G\eta| \leq l$. Suppose for a contradiction that $g_0\eta,\dots,g_l\eta$ are pairwise distinct for some $g_0,\dots,g_l \in G$.  Pick $(y_{i,j})_{i<j \leq l}$ witnessing it, i.e.\ for every $i<j\leq l$ we have $g_i \eta(y_{i,j}) \ne g_j \eta(y_{i,j})$. Choose $U_{i,j}$ clopen so that $g_i \eta(y_{i,j}) \in U_{i,j}$ and $g_j \eta(y_{i,j}) \in U_{i,j}^c$ for all $i<j\leq l$. This is an open condition on $\eta$, so there are $\bar y \in X^n$, $\bar{U} \in \CB^n$, and finite $A \subseteq G$ such that:
\begin{itemize}
\item $\{g_0^{-1},\dots,g_l^{-1}\} \subseteq A$,
\item $\{(y_{i,j},U_{i,j}): i<j \leq l\} \subseteq \{(y_s,U_s): s <n\}$,
\item $g_i g_{\bar{y},\bar{U},A}^{-1}y_{i,j} \in U_{i,j}$ and $g_j g_{\bar{y},\bar{U},A}^{-1}y_{i,j} \in U_{i,j}^c$ for all $i<j\leq l$.
\end{itemize}
These conditions imply that $|c[g_{\bar{y},\bar{U},A} A]| > l$, a contradiction.

(2) $\Rightarrow$ (1) By (2), $G\eta=\{g_0\eta,\dots,g_{l-1}\eta\}$ for some $g_0,\dots,g_{l-1} \in G$. Consider an arbitrary externally definable coloring $c \colon G \to 2^n$. It is given by 
$$c(g)(s)=
\begin{cases}
1 & \mbox{ if }y_s\in gU_s\\
0 & \mbox{ if } y_s\notin gU_s
\end{cases}$$
for some $y_0,\dots,y_{n-1} \in X$ and $U_0,\dots,U_{n-1} \in \CB$. Consider any finite $A=\{a_0,\dots,a_{k-1}\} \subseteq G$. Then for every $i<k$ there exists $j(i)<l$ such that $a_i^{-1}\eta =g_{j(i)} \eta$, and so for every $s<n$ we have $a_i^{-1}\eta(y_s) =g_{j(i)} \eta(y_s)$. Hence, 
for all $i<k$ and $s<n$ we have $a_i^{-1}\eta(y_s) \in U_s$ iff $g_{j(i)} \eta(y_s) \in U_s$, which is an open condition on $\eta$. Therefore, there is $h \in G$ such that for all $i<k$ and $s<n$ we have $y_s \in ha_iU_s$ iff $y_s \in hg_{j(i)}^{-1}U_s$. So $c(ha_i) = c(hg_{j(i)}^{-1})$ for all $i<k$. Thus, $|c[hA]| \leq c[h\{g_0^{-1},\dots,g_{l-1}^{-1}\}]| \leq l$. 
\end{proof}

\begin{definition}\label{definition: joint. fin. EDERdeg with respect to Sigma}
Let $\Sigma \subseteq \CB$ (possibly infinite). We say that the $G$-ambit $X$ has {\em jointly finite externally definable  Ramsey degree with respect to $\Sigma$} (joint.\ fin.\ EDRdeg with respect to $\Sigma$) if there exists $l \geq 1$ such that for every externally definable $\Sigma$-coloring $c \colon G \to 2^n$, and for every $k \geq 1$ and $A \in [G]^k$, there exists $h \in G$ such that $|c[hA]| \leq l$.
\end{definition}

For a subset $\Delta \subseteq \CB$ let $\CA(\Delta)$ be the Boolean algebra (this time not necessarily $G$-algebra) of subsets of $G$ generated by $\Delta$. Set $E(X) |^*_{\Delta}:=\{ \eta |^*_{\Delta} : \eta \in E(X)\}$, where $\eta |^*_{\Delta} \colon X \to S(\CA(\Delta))$ is given by $\eta |^*_{\Delta}(p):= \eta(p) \cap \CA(\Delta)=\{U \in \CA(\Delta): U \in \eta(p)\}$ (bearing in mind Remark \ref{rem: identifying_S(X)}, we identify here $X$ with $S(X):=S(\mathcal{B})$). We equip $E(X) |^*_{\Delta}$ with the product topology (where $S(\CA(\Delta))$ is equipped with the Stone space topology). It is a compact space, but there is no natural $G$-action on $E(X) |^*_{\Delta}$. Note that the assignment $\eta \mapsto \eta |^*_{\Delta}$ yields a continuous surjection $\pi_{\Delta} \colon E(X) \to E(X) |^*_{\Delta}$. For any $\Delta_0 \subseteq \Delta_1 \subseteq \CB$, we have $\pi^{\Delta_1}_{\Delta_0} \colon E(X) |^*_{\Delta_1} \to E(X) |^*_{\Delta_0}$ given by $\pi^{\Delta_1}_{\Delta_0}(\eta |^*_{\Delta_1}):= \eta |^*_{\Delta_0}$ which is a continuous surjection so that $(E(X) |^*_{\Delta}, \pi^{\Delta_1}_{\Delta_0})_{\Delta_0 \subseteq \Delta_1 \subseteq \CB}$ forms an inverse system.

Now, let $\mathcal{D}$ be a family of subsets of $\CB$. We say that it is {\em cofinal} if each $U \in \CB$ belongs to some $\Delta$ from $\mathcal{D}$. If $\mathcal{D}$ is upward directed by inclusion, we can and will index it as $\{\Delta_i\}_{i \in I}$ for a directed set $(I,\leq)$ so that $i \leq i'$ iff $\Delta_i \subseteq \Delta_{i'}$.

\begin{remark}\label{remark: E(X) as an inverse limit}
Let $\mathcal{D}=\{\Delta_i\}_{i \in I}$ be a cofinal, upward directed by inclusion family of subsets of $\CB$. Then the map $\Phi \colon E(X) \to \invlim_{i} E(X) |^*_{\Delta_i}$ given by $\Phi(\eta) := (\eta |^*_{\Delta_i})_i$ is a homeomorphism. (In particular, this induces a $G$-ambit structure on the compact space $\invlim_{i} E(X) |^*_{\Delta}$.) 
\end{remark}

The next lemma yields a criterion for a natural presentation of some left ideal in $E(X)$ as an inverse limit of finite spaces.

\begin{lemma}\label{lemma: main lemma for metrizability}
Assume that there is a cofinal, upward directed by inclusion family $\mathcal{D}=\{\Delta_i\}_{i \in I}$ of subsets of $\CB$ such that for every $\Delta \in \mathcal{D}$ the $G$-ambit $X$ has joint.\ fin.\ EDRdeg with respect to $\Delta$. Then there exists $\eta \in E(X)$ such that for every $i \in I$ the subspace $(G\eta) |^*_{\Delta_i}:=\{(g\eta)|^*_{\Delta_i}: g \in G\}$ of $E(X) |^*_{\Delta_i}$ is finite and $\Phi[E(X)\eta]=\invlim_i (G\eta) |^*_{\Delta_i}$, where $\Phi$ is the homeomorphism from Remark \ref{remark: E(X) as an inverse limit}.
\end{lemma}

\begin{proof}
For each $i \in I$, the proof of (1) $\Rightarrow$ (2) in Proposition \ref{proposition: joint fin EDERdeg} (in which all the clopens should be now taken from $\Delta_i$ instead of the whole $\CB$) yields $\eta_i \in E(X)$ with $(G \eta_i) |^*_{\Delta_i}$ finite.

For every $i \in I$ we have $(E(X)\eta_i)|^*_{\Delta_i} = (\cl(G\eta_i))|^*_{\Delta_i} = \pi_{\Delta_i}[\cl(G\eta_i)] = \cl(\pi_{\Delta_i}[G\eta_i])= \cl((G \eta_i) |^*_{\Delta_i}) = (G\eta_i) |^*_{\Delta_i}$, where the third equality follows from continuity of $\pi_{\Delta_i}$ and the last one from finiteness of $(G\eta_i) |^*_{\Delta_i}$.

For any finite $\bar{i}=(i_0,\dots,i_{n-1}) \in I^n$ set $\eta_{\bar i}:=\eta_{i_0} \circ \dots \circ \eta_{i_{n-1}}$. For every $j<n$, we have $|(G\eta_{\bar{i}})|^*_{\Delta_{i_j}}| \leq |(G(\eta_{i_0} \circ \dots \circ \eta_{i_j}))|^*_{\Delta_{i_j}}| \leq |(E(X)\eta_{i_j})|^*_{\Delta_{i_j}}|= |(G\eta_{i_j}) |^*_{\Delta_{i_j}}|$, where the last equality follows from the last paragraph.

Order $\bigcup_{n} I^n$ by: $\bar{i} \leq \bar{i'}$ iff $\bar{i}$ is a subtuple of $\bar{i'}$. Let $\eta$ be an accumulation point of the net $(\eta_{\bar i})_{\bar{i}}$. We claim that for every $i \in I$ the set $(G\eta) |^*_{\Delta_i}$ is finite of size at most $s_i:=|(G \eta_i) |^*_{\Delta_i}|$. Suppose for a contradiction that $|(G\eta) |^*_{\Delta_i}| > s_i$. Then there are $g_0,\dots,g_{s_i} \in G$ with $(g_j\eta)|^*_{\Delta_i} \ne (g_{j'}\eta)|^*_{\Delta_i}$ for all $j<j' \leq s_i$. So there are $p_{j,j'}$ for $j<j' \leq s_i$ such that $(g_j\eta)|^*_{\Delta_i}(p_{j,j'}) \ne (g_{j'}\eta)|^*_{\Delta_i}(p_{j,j'})$ for all $j<j'\leq s_i$.
This is an open condition on $\eta$, so there is $\bar{i}$ in some $I^n$ such that $i$ is one of the coordinates of $\bar{i}$ and $(g_j\eta_{\bar{i}})|^*_{\Delta_i}(p_{j,j'}) \ne (g_{j'}\eta_{\bar{i}})|^*_{\Delta_i}(p_{j,j'})$ for all $j<j'\leq s_i$. Thus, $|(G\eta_{\bar{i}}) |^*_{\Delta_i}| >s_i$ which contradicts the previous paragraph.

We clearly have $\Phi[G\eta]=\{ ((g \eta)|^*_{\Delta_{i}})_{i \in I}: g \in G\}$ which is dense in $\invlim_i (G \eta)|^*_{\Delta_i}$. Since $\Phi$ is continuous and  $\invlim_i (G \eta)|^*_{\Delta_i}$ is closed in $\invlim_{i} E(X) |^*_{\Delta_i}$ by finiteness of all $(G \eta)|^*_{\Delta_i}$, we conclude that $\Phi[E(X)\eta]=\Phi[\cl(G \eta)]=\invlim_i (G\eta) |^*_{\Delta_i}$.
\end{proof}

The main conclusion is the following Ramsey-theoretic criterion of metrizability of the minimal left ideals in $E(X)$.

\begin{corollary}\label{corollary: criterion for metrizability}
Assume that there is a {\em countable} cofinal, upward directed by inclusion family $\mathcal{D}=\{\Delta_i\}_{i \in I}$ of subsets of $\CB$ such that for every $\Delta \in \mathcal{D}$ the $G$-ambit $X$ has joint.\ fin.\ EDRdeg with respect to $\Delta$. Then every minimal left ideal in $E(X)$ is metrizable.
\end{corollary}

\begin{proof}
Since all minimal left ideals in $E(X)$ are homeomorphic (even isomorphic as $G$-flows; cf.\ \cite[Proposition I.2.5]{Gla}), it is enough to show that some minimal left ideal is metrizable. Take $\eta$ provided by Lemma \ref{lemma: main lemma for metrizability} and pick a minimal left ideal $\mathcal{M}$ contained in the left ideal $E(X)\eta$. If we show that $E(X) \eta$ is metrizable, then so is $\mathcal{M}$, and we will be done. By the choice of $\eta$, $E(X)\eta$ is homeomorphic with $\invlim_{i\in I} (G\eta) |^*_{\Delta_i}$. Since $I$ is countable and all $ (G\eta) |^*_{\Delta_i}$'s are finite, the last inverse limit is compact and has a countable basis of open sets, so it is metrizable.
\end{proof}

\subsubsection{WAP ambits and examples}\label{subsection: examples}

Big families of examples of $0$-dimensional ambits with sep.\ fin.\ EDRdeg appear naturally in the context of first order theories \cite{KLM}, which will be discussed in Subsection \ref{subsection: first order theories}. They come from KPT theory and also from stable theories. In Subsection \ref{subsection: KPT context}, we will also see that the specialization to the classical KPT context provides lots of examples. 

In this section, we extend the family of ambits obtained for stable theories to the general context of $0$-dimensional  ambits which are WAP ({\em weakly almost periodic}, see \cite[Section II]{ElNe} for the definition and basic facts), i.e.\ we show that they have sep.\ fin.\ EDRdeg. Although the most important notion is that of having sep.\ fin.\ EDRdeg, one could wonder how restrictive the notion of having sep.\ fin.\ DRdeg is. At the end of this subsection, we outline an example of a $0$-dimensional ambit which does not have sep.\ fin.\ DRdeg (whereas, as mentioned after Definition \ref{definition: sep. fin EDERdeg}, every $0$-dimensional ambit has sep.\ fin.\ WDRdeg which makes this property useless). Surprisingly, we do not know any example of a $0$-dimensional ambit with sep.\ fin.\ DRdeg but without sep.\ fin.\ EDRdeg (see Question \ref{question: sep. fin DERdeg vs sep. fin. EDERdeg}).

\begin{proposition}\label{proposition: WAP implies finite RP}
If the $G$-ambit $X$ is WAP, then it has sep.\ fin.\ EDRdeg.
\end{proposition}

\begin{proof}
By Corollary \ref{corollary: characterization_sep_fin_EDERdeg}, we need to show that for each finite subset $\Delta$ of $G\backslash \CB$ there exists $\eta\in E(S_{\Delta}(X))$ with finite image.

WAP implies that there is a unique minimal left ideal $I$ in $E(X)$ and $I$ is a compact (topological) group (see \cite[Proposition II.5]{ElNe}). Since $I$ is a group, all elements of $I$ have the same image.

\begin{clm}
$M:=Ix_0$ is a unique minimal subflow of $(G,X)$.
\end{clm}

\begin{clmproof}
 Minimality of $M$ follows from minimality of $I$. For uniqueness consider any minimal subflow $N$ of $X$. Let $\ev_{x_0}\colon E(X) \to X$ be the evaluation at $x_0$. This is a an epimorphism of $G$-flows. So $\ev_{x_0}^{-1}[N]$ is a subflow of $E(X)$ and hence we can find a minimal left ideal $J$ in $E(X)$ which is contained in $\ev_{x_0}^{-1}[N]$. Then $J=I$ by uniqueness of a minimal left ideal. Thus, $N=\ev_{x_0}[\ev_{x_0}^{-1}[N]] \supseteq Jx_0=Ix_0=M$. So $N=M$ by minimality of $N$. 
\end{clmproof}

\begin{clm}\phantomsection
\begin{enumerate}
\item For every $x \in X$ we have $Ix =M$.
\item For every $\eta \in I$ we have $\im(\eta)=M$.
\end{enumerate}
\end{clm}

\begin{clmproof}
(1) follows from Claim 1 and minimality of $I$.

(2) The inclusion $\im(\eta)\subseteq M$ is immediate by (1). For the opposite inclusion consider any $x \in M$. By (1), it can be written as $\alpha(x)$ for some $\alpha \in I$. Since $I$ is a group, $\alpha=\eta(\eta^{-1} \alpha)$, so $x = \eta(\eta^{-1}\alpha(x)) \in \im(\eta)$. 
\end{clmproof}

The next straightforward claim is a general remark. We use here the notion introduced at the beginning of Subsection \ref{subsubsection: Ramsey and profiniteness}, in particular, the function $\pi_\Delta \colon E(X) \to S_{\Delta}(X)$ introduced after Remark \ref{rem: identifying_S(X)}.

\begin{clm}\phantomsection
\begin{enumerate}
\item Let $Y$ be a subflow of the $G$-flow $X$. Then the function $r_Y \colon \pi_\Delta[Y] \to S(\{A \cap Y: A \in \mathcal{A}(\Delta)\})$ given by  $r_Y(p):=\{A \cap Y: A \in p\}$ is an isomorphism of $G$-flows.
\item Let $\Pi_{\Delta} \colon E(X) \to E(S_{\Delta}(X))$ be the epimorphism induced by $\pi_\Delta$ via $(\Pi_\Delta(\eta))(\pi_\Delta(p)):=\pi_\Delta(\eta(p))$. Then $\im(\Pi_\Delta(\eta)) = \pi_\Delta[\im(\eta)]$ for any $\eta \in E(X)$.
\end{enumerate}
\end{clm}

Take any $\eta \in I$. By Claim 2(2) and Claim 3(2), $\im(\Pi_\Delta(\eta)) = \pi_\Delta[M]$. On the other hand, by Claim 3(1), $|\pi_\Delta[M]|=|S(\{A \cap M: A \in \mathcal{A}(\Delta)\})|$. It remains to show that $\{A \cap M: A \in \mathcal{A}(\Delta)\}$ is finite, because then  $\Pi_\Delta(\eta)$ is an element of $E(S_\Delta(X))$ with finite image.

$(G,M)$ is a WAP minimal flow and so equicontinuous by \cite[Proposition II.8]{ElNe}. Thus, $E(M)$ is a compact group and $(G,M) \cong (G,E(M)/H)$ for some closed subgroup $H$ of $E(M)$ (see \cite[Theorem I.3.3(5)]{Gla}). Let $\theta$ be a witnessing isomorphism of $G$-flows.

We can write $\Delta=\{GU_0,\dots,GU_{n-1}\}$ for some clopens $U_0,\dots,U_{n-1}$ in $X$. Then $\CA(\Delta)$ consists of the Boolean combinations of the $G$-translates of $U_0,\dots,U_{n-1}$. Thus, the Boolean $G$-algebra $\{A \cap M: A \in \mathcal{A}(\Delta)\}$ consists of the Boolean combinations of the $G$-translates of $U_0\cap M,\dots,U_{n-1}\cap M$. 

Each $U_i \cap M$ is a clopen subset of $M$, so $\theta[U_i \cap M]$ is a clopen subset of $E(M)/H$. If  $U_i \cap M \ne \emptyset$, then, by compactness of $E(M)$ and closedness of $H$, the set $\theta[U_i \cap M]$ is of the form $C/H$ where $C$ is a union of finitely many left cosets of some clopen (so finite index) subgroup of $E(M)$ containing $H$. Hence, $\theta[U_i \cap M]$ has only finitely many $G$-translates, and so does $U_i \cap M$.

By the last two paragraphs, the Boolean algebra $\{A \cap M: A \in \mathcal{A}(\Delta)\}$ is generated by finitely many sets, and so it is finite, as required.
\end{proof}

\begin{lemma}\label{lemma: criterion for the lack of sep fin DERdeg}
Let $G$ be a group and $U \subseteq G$ a subset such that for every $k \in \mathbb{N}$ there are $a_0,\dots,a_{k-1} \in G$ such that for every $i \ne j$ we have that $a_iU \cap a_jU^c$ is right syndetic. Let $\mathcal{A}$ be any $G$-algebra of subsets of $G$ which contains $U$. Then the ambit $(G,S(\mathcal{A}),p_e)$ does not have sep.\ fin.\ DRdeg.
\end{lemma}

\begin{proof}
Suppose for a contradiction that $(G,S(\mathcal{A}),p_e)$ has sep.\ fin.\ DRdeg. Then, for $\Sigma:=\{U\}$ there is $l=l(Gp_e,\Sigma)$ such that for every definable $\Sigma$-coloring $c \colon G \to 2^n$ and for every $k \geq 1$ and $A \in [G]^k$, there exists $h \in G$ with $|c[hA]| \leq l$.

Take $k=l+1$ and elements $a_0,\dots,a_l \in G$ such that for every $j< j' \leq l$ the set $a_jU \cap a_{j'}U^c$ is right syndetic. Then there are $g_0,\dots, g_{n-1} \in G$ such that for every $j< j' \leq l$ we have 
$$(a_jU \cap a_{j'}U^c)g_0^{-1} \cup \dots \cup (a_jU \cap a_{j'}U^c)g_{n-1}^{-1} = G.$$

Put $A:=\{a_0,\dots,a_l\}$. Let $c \colon G \to 2^n$ be given by
$$c(g)(i)=
\begin{cases}
1 & \mbox{ if }g_i\in gU\\
0 & \mbox{ if } g_i\notin gU
\end{cases}.$$
This is a definable $\Sigma$-coloring. From the previous paragraph, we get that for every $h \in G$ and for every $j < j' \leq l$, there is $i<n$ with $c(ha_{j})(i) \ne c(ha_{j'})(i)$, and hence $|c[hA]| \geq l+1$ which contradicts the choice of $l$. 
\end{proof}

\begin{example}\label{example: not sep. fin. DERdeg}
The ambit $(\mathbb{Z},\beta \mathbb{Z}, p_0)$ does not have sep.\ fin.\ DRdeg.
\end{example}

\begin{proof}
Consider the action of $\mathbb{Z}$ on the unit circle $S^1$ identified with $\mathbb{R}/\mathbb{Z}$ given by $n \cdot (z/\mathbb{Z}):=(n\alpha + z)/\mathbb{Z}$, where $\alpha$ is an irrational number. Then $(\mathbb{Z},S^1)$ is a minimal flow.  

For every integer $n \ne 0$ we have $n\alpha \notin \mathbb{Z}$, so $A_n:=([0,\frac{1}{2}) \setminus ([0,\frac{1}{2}) + n\alpha))/\mathbb{Z}$ has nonempty interior. Thus, by minimality of $(\mathbb{Z},S^1)$, the set $U_n:=\{ m \in \mathbb{Z}: m\alpha/\mathbb{Z} \in A_n\}$ is syndetic.

Put $U:=\{m \in \mathbb{Z}: m\alpha/\mathbb{Z} \in [0,\frac{1}{2})/\mathbb{Z}\}$. Note that $U_n=U \cap (n + U^c)$. Thus, for any integers $n \ne n'$ we have that $(n+U) \cap (n' + U^c) = n + (U \cap ((n'-n) +U^c)) = n+U_{n'-n}$ is syndetic by the last paragraph. So the conclusion follows from Lemma \ref{lemma: criterion for the lack of sep fin DERdeg}.
\end{proof}

Combining Examples \ref{example: WDERP does not imply DERP} and \ref{example: not sep. fin. DERdeg}, one can get the following example showing that WDRP does not imply sep.\ fin.\ DRdeg. The proof is left as an exercise.

\begin{example}
Take $U \subseteq \mathbb{Z}$ so that $U \cap (n+U^c)$ is syndetic for every nonzero integer $n$ (e.g.\ the one constructed in the proof of  Example \ref{example: not sep. fin. DERdeg}). Let $f \colon F_{a,b} \to \mathbb{Z}$ be a unique epimorphism such that $f(a)=1$ and $f(b)=0$. Define  
$$V:=\{ x \in F_{a,b}: \textrm{the reduced word of $x$ ends in $a$ and } f(x) \in U\}.$$ 
Then $V$ is not left syndetic but $V \cap a^nV^c$ is right syndetic for every integer $n \ne 0$. Thus, for the Boolean $F_{a,b}$-algebra $\mathcal{A}$ of subsets of $F_{a,b}$ generated by $V$, we get that $(F_{a,b},S(\mathcal{A}),p_e)$ has WDRP but not sep.\ fin.\ DRdeg.
\end{example}

We now give a generalization of Example \ref{example: not sep. fin. DERdeg}.

\begin{proposition}\label{proposition: No sep. fin DERdeg}
For every infinite abelian group $G$ the ambit $(G,\beta G,p_e)$ does not have sep.\ fin.\ DRdeg.
\end{proposition}

\begin{proof}
Consider the universal minimal $G$-flow $X:=M(G)$ contained in $\beta G$.

\begin{clm}
There exists a clopen subset $U_0$ of $X$ with infinite $G$-orbit.
\end{clm}

\begin{clmproof}
Suppose every clopen $U_0 \subseteq X$ has finite orbit. Then for every finite $\Delta \subseteq G \backslash \CB$ (i.e.\ a finite collection of orbits of clopens in $X$) the $G$-algebra $\CA(\Delta)$  generated by $\bigcup \Delta$ is finite, and so is its Stone space $S_{\Delta}(X)$. Since by Remark \ref{rem: identifying_S(X)} $X$ is isomorphic to  the inverse limit of the finite flows $S_{\Delta}(X)$ for $\Delta$ ranging over the finite subsets of $G \backslash \CB$, the flow $(G,X)$ is a profinite flow. Since all finite flows are equicontinuous, we conclude that so is $X$. 
Thus, by \cite[Theorem I.3.3(5)]{Gla}, $X$ is a compact homogeneous space. On the other hand, $X=M(G)$ is an extremally disconnected space (e.g.\ see \cite{Dou}). Finally, every compact homogeneous extremally disconnected space is finite (see \cite[Corollary 4]{Rez}). Thus, we conclude that $X=M(G)$ is finite. But this contradicts non-metrizability of $M(G)$ which holds by \cite[Theorem A2.2]{KPT}.
\end{clmproof}

\begin{clm}
For every $g,g' \in G$ such that $gU_0 \ne g'U_0^c$ we have $gU_0 \cap g'U_0^c \ne \emptyset$.
\end{clm}

\begin{clmproof}
Suppose $gU_0 \cap g'U_0^c = \emptyset$. Then $gU_0 \subsetneq g'U_0$, so $U_0 \subsetneq g^{-1}g'U_0$. By amenability of $G$, there is a regular Borel  probability measure $\mu$ on $X$ invariant under $G$. Since $g^{-1}g'U_0 \setminus U_0 \ne \emptyset$ is open and $\mu$ is invariant, we have $\mu(g^{-1}g'U_0 \setminus U_0)>0$, and $\mu(g^{-1}g'U_0)=\mu(U_0) +\mu(g^{-1}g'U_0 \setminus U_0) = \mu(g^{-1}g'U_0) + \mu(g^{-1}g'U_0 \setminus U_0)>\mu(g^{-1}g'U_0)$, a contradiction.
\end{clmproof}

By Claims 1 and 2, there exists an infinite sequence $(g_i)_{i \in \omega}$ of elements of $G$ such that $g_iU_0 \cap g_jU_0^c \ne \emptyset$ for all $i \ne j$. Now, pick any $x_0 \in X$, and consider any $i \ne j$ from $\omega$. Since $g_iU_0 \cap g_jU_0^c$ is a nonempty open subset of the minimal $G$-flow $X$, we easily get that $U_{ij}:=\{g\in G: gx_0 \in g_iU_0 \cap g_jU_0^c \}$ is left syndetic, equivalently right syndetic as $G$ is abelian. An elementary computation shows that $U_{ij}= g_iU \cap g_jU^c$, where $U:=\{g \in G: gx_0 \in U_0\}$. So we have shown that $U$ satisfies the assumption of Lemma \ref{lemma: criterion for the lack of sep fin DERdeg}, and hence $(G,\beta G,p_e)$ does not have sep.\ fin.\ DRdeg.
\end{proof}

A natural question arises if the abelianity assumption can be generalized to amenability. Finally, one can ask if amenability can be dropped:

\begin{question}\label{question: sep. fin. EDERdeg for beta G}
Is it true that for every infinite group $G$ the ambit $(G,\beta G,p_e)$ does not have sep.\ fin.\ DRdeg.
\end{question}

It is well-known (e.g.\ follows immediately from \cite[Theorem A2.2]{KPT}), that an infinite discrete group is not extremely amenable, which means that the ambit $(G, \beta G, p_e)$ does not have a fixed point, equivalently it does not have WDRP. A positive answer to the above question would be a strengthening of this result, because the negation of sep.\ fin.\ DRdeg implies the negation of DRP which for the ambit $(G,\beta G,p_e)$ is equivalent to the negation of WDRP by Corollary \ref{corollary: right invariant G-algebras}.

We also state

\begin{question}\label{question: sep. fin DERdeg vs sep. fin. EDERdeg}
Does sep.\ fin.\ DRdeg imply sep.\ fin.\ EDRdeg (which would mean that they are equivalent)?
\end{question}

\section{Specializations}\label{section: specializations}

In this section, we describe and study three specializations of the abstract context developed in Section \ref{section: main}.

In Subsection \ref{subsection: first order theories}, we explain that in the context of first order theories the various notions of definable colorings and definable Ramsey properties as well as the main results from Section  \ref{section: main} (except those from Subsection \ref{subsection: joint. fin. EDERdeg}) specialize to the corresponding notions and results from \cite{KLM}. This was the original reason to define the notions from Section  \ref{section: main} in the way we did, and the original motivation was to make the ideas from \cite{KLM} fully understandable to the readers who are not familiar with model theory by adapting them to a purely topological dynamics context. The abstract context from Section \ref{section: main} specializes also to definable groups and to classical KPT theory (which will be described in Subsections \ref{subsection: definable groups} and \ref{subsection: KPT context}), which makes the subject even more interesting.

\subsection{First order theories}\label{subsection: first order theories}

We recall the notions of definable colorings from \cite{KLM}. The reader is advised to read the short introductory part of Section 3 in \cite{KLM} where [externally] definable colorings are defined and discussed, including an explanation of terminology.

Let $T$ be a complete first order theory. Let $\mathfrak{C}$ be a monster model of $T$ and $\bar c$ an enumeration of $\C$. For a tuple of variables $\bar y$ by $S_{\bar{y}}(\C)$ we denote the space of complete types over $\C$ in variables $\bar y$; $S_{\bar{c}}(\C)$ denotes the space of complete types over $\C$ extending $\tp(\bar{c}/\emptyset)$.

For a finite tuple $\bar a\in \C$ and a subset $C$ of $\C$, we write ${C\choose\bar a}$ for the set of realizations of $\tp(\bar a)$ which are contained in $C$, that is, 
$${C\choose\bar a}:=\{\bar a'\in C^{|\bar a|} : \bar a'\equiv \bar a\}.$$


\begin{definition}\phantomsection\label{definition: of definable colorings}\cite[Definition 3.1]{KLM}
\begin{enumerate}
\item A coloring $c\colon {C\choose\bar a}\to 2^n$ is {\em definable} if there are formulas with parameters $\varphi_i(\bar x)$, $i<n$, such that:
$$c(\bar a')(i)= \left\{
\begin{array}{cl}
1, & \models \varphi_i(\bar a')\\
0, & \models \lnot\varphi_i(\bar a')
\end{array}
\right.$$
for any $\bar a'\in{C\choose \bar a}$ and $i<n$. 
\item A coloring $c:{C\choose\bar a}\to 2^n$ is {\em externally definable} if there are formulas without parameters $\varphi_i(\bar x,\bar y)$ and types $p_i(\bar y)\in S_{\bar y}(\C)$, $i<n$, such that:
$$c(\bar a')(i)= \left\{
\begin{array}{cl}
1, &  \varphi_i(\bar a',\bar y)\in p_i(\bar y)\\
0, & \lnot\varphi_i(\bar a',\bar y)\in p_i(\bar y)
\end{array}
\right.$$
for any $\bar a'\in{C\choose\bar a}$ and $i<n$. 
\item If $\Delta$ is a set of formulas in variables $\bar x$ and $\bar y$ and $q \in S_{\bar y}(\emptyset)$, then an externally definable coloring $c$ is called {\em externally definable $(\Delta,q)$-coloring} if all the formulas $\varphi_i(\bar x,\bar y)$'s defining $c$ are taken from $\Delta$ and $p_i(\bar y) \in S_{q}(\C)$ for $i<n$.
\end{enumerate}
\end{definition}

\begin{remark}\phantomsection\label{remark: ext def coloring}\cite[Remark 3.3]{KLM}
\begin{enumerate}
\item An externally definable coloring $c\colon {C\choose\bar a}\to 2^n$ given by $\varphi_i(\bar x,\bar y)$ and $p_i(\bar y)\in S_{\bar y}(\C)$, $i<n$, can be defined by using $n$ formulas $\psi_i(\bar x,\bar z)$, $i<n$, and only one type $p(\bar z)\in S_{\bar z}(\C)$. Moreover, we can require that $|\bar z|=|\bar c|$ and $p(\bar z) \in S_{\bar c}(\C)$. 
\item In the definition of definable coloring, we can assume that all formulas $\varphi_i(\bar x)$, $i<n$, have the same parameters $\bar e$, and then the coloring is externally definable witnessed by the single realized type $p(\bar y) := \tp(\bar e/\C)$. We can also assume that $\bar e = \bar c$.
\item In the definition of externally definable $(\Delta,q)$-coloring, without loss of generality we can assume that $|\bar y|=|\bar c|$ and $q(\bar y)=\tp(\bar c/\emptyset)$. Whenever $|\bar y|=|\bar c|$ and $q(\bar y) = \tp(\bar c/\emptyset)$, instead of ``externally definable $(\Delta,q)$-coloring'' we will just say {\em externally definable $\Delta$-coloring}.
\end{enumerate}
\end{remark}

We now describe a correspondence between [externally] definable colorings of the copies of a finite tuple from $\C$ in the above sense and [externally] definable colorings of elements of $\Aut(\C)$ with respect to the ambit $(\Aut(\C),S_{\bar c}(\C),\tp(\bar c/\C))$ in the sense of Definitions \ref{definition: definable colorings} and \ref{definition: Sigma-colorings}.

Consider any [externally] definable coloring $c\colon {\C\choose\bar a}\to 2^n$. By Remark \ref{remark: ext def coloring}, it is defined by formulas $\varphi_i(\bar x,\bar y)$ with $|\bar y| = |\bar c|$ and types $p_i(\bar y) \in S_{\bar c}(\C)$, where: if $c$ is definable, the types $p_i$ are realized; if $c$ is an externally definable $\Delta$-coloring, the formulas $\varphi_i(\bar x,\bar y)$ belong to $\Delta$. With $c$ we associate a coloring $c' \colon \Aut(\C) \to 2^n$ via 
$$c'(\sigma)(i) := c(\sigma(\bar a))(i).$$

\begin{remark}\label{remark: c to c'}
If $c$ is an [externally] definable coloring or an externally definable $\Delta$-coloring, then $c'$ is respectively an [externally] definable coloring or an externally definable $\Delta'$-coloring, where $\Delta':=\{[\varphi(\bar a,\bar y)]: \varphi(\bar x,\bar y) \in \Delta\}$.
\end{remark}

\begin{proof}
It is straightforward: for clopens $U_0,\dots,U_{n-1}$ in Definition \ref{definition: definable colorings} we take the sets $[\varphi_0(\bar a,\bar y)], \dots,[\varphi_{n-1}(\bar a,\bar y)]$.
\end{proof}

In the opposite direction, consider any [externally] definable coloring $c' \colon \Aut(\C) \to 2^n$ in the sense of Definition \ref{definition: definable colorings}. It is defined by some clopens $U_0=[\varphi_0(\bar a,\bar y)], \dots U_{n-1} = [\varphi_{n-1}(\bar a,\bar y)]$ and types $q_0(\bar y),\dots,q_{n-1}(\bar y) \in S_{\bar c}(\C)$. If $c'$ is an externally definable $\Delta'$-coloring for some finite collection $\Delta'=\{V_0,\dots, V_{m-1}\}$ of clopens in $S_{\bar c}(\C)$, then we require that $U_0,\dots,U_{n-1} \in \Delta'$ and we choose $\Delta=\{\psi_0(\bar a,\bar y),\dots,\psi_{m-1}(\bar a, \bar y)\}$ so that $V_i = [\psi_i(\bar a, \bar y)]$ for all $i <m$.\footnote{Note that for an infinite $\Delta'$ we would not be able to find a common finite tuple $\bar a$ so that each $U \in \Delta'$ equals $[\psi(\bar a, \bar y)]$ for some formula $\psi$.}
With $c'$ we associate a coloring  $c\colon {\C\choose\bar a}\to 2^n$ via 
$$c(\sigma(\bar a))(i) := c'(\sigma)(i).$$

\begin{remark}\label{remark: c' to c}
If $c'$ is an [externally] definable coloring or an externally definable $\Delta'$-coloring for some finite collection $\Delta'$ of clopens in $S_{\bar c}(\C)$, then $c$  is respectively an [externally] definable coloring or an externally definable $\Delta$-coloring.
\end{remark}

\begin{proof}
Straightforward.
\end{proof}

\begin{definition}\phantomsection\cite[Definitions 3.11 and 3.18]{KLM}\label{definition: EDEERP  and sep. fin. EDEERdeg from KLM}
\begin{enumerate}
	\item The theory $T$ has {\em [E]DEERP} if for any finite $B \subseteq \C$ and finite tuple $\bar a\subseteq B$, any $n<\omega$ and any [externally] definable coloring $c\colon {\C\choose \bar a}\rightarrow 2^n$, there exists $B' \in {\C\choose B}$ (i.e., a copy of $B$ under an automorphism of $\C$) such that $|c[{B'\choose \bar a}]|=1$.
	
	\item The theory $T$ has {\em separately finite EDEERdeg} if for any finite tuple $\bar a$ and finite set of formulas $\Delta$ there exists $l=l(\bar a,\Delta)<\omega$ such that for any finite $B\subseteq \C$ containing $\bar a$, for any $n<\omega$, and for any externally definable $\Delta$-coloring $c\colon{\C\choose \bar a}\rightarrow 2^n$, there exists $B'\in{\C\choose B}$ such that $|c[{B'\choose \bar a}]|\le l$.
\end{enumerate}
\end{definition}

\begin{proposition}\phantomsection\label{proposition: specialization of EDERP to theories} 
\begin{enumerate}
	\item The theory $T$ has [E]DEERP if and only if the ambit $(\Aut(\C),S_{\bar c}(\C), \tp(\bar c/\C))$ has [E]DRP.
	\item The theory $T$ has sep.\ fin.\ EDEERdeg if and only if the ambit $(\Aut(\C),S_{\bar c}(\C),\tp(\bar c/\C))$ has sep.\ fin.\ EDRdeg.
\end{enumerate}
\end{proposition}

\begin{proof}
It essentially follows from the above correspondence between colorings and Remarks \ref{remark: c to c'} and \ref{remark: c' to c}. We sketch the proof of (1), leaving (2) as an exercise. 

($\Leftarrow$) Given an [externally] definable coloring $c$ of copies of $\bar a$ and a finite $B \subseteq \C$ containing $\bar a$, the set ${B\choose \bar a}$ equals $\{\sigma_0(\bar a),\dots, \sigma_{k-1}(\bar a)\}$ for some $\sigma_0,\dots, \sigma_{k-1} \in \Aut(\C)$ which yield a finite set $A :=\{\sigma_0,\dots, \sigma_{k-1}\} \subseteq \Aut(\C)$. Then, by [E]DRP, for the associated coloring $c'$ we find $\sigma \in \Aut(\C)$ such that $|c'[\sigma A]|=1$. This implies that for $B':=\sigma[B]$ we have $|c[{B'\choose \bar a}]|=1$.

($\Rightarrow$) Given an [externally] definable coloring $c'$ of $G$ and a finite $A \subseteq \Aut(\C)$, we find a finite tuple $\bar a$ and the associated coloring $c$ of the copies of $\bar a$ as described above. Set $B$ to be the union of all coordinates of all tuples $\sigma(\bar a)$ for $\sigma \in A$. By [E]DEERP, there is a copy $B'$ of $B$ via some $\sigma \in \Aut(\C)$ such that $|c[{B'\choose \bar a}]|=1$. This implies that $|c'[\sigma A]|=1$.
\end{proof}

Similarly, the condition that $(\Aut(\C),S_{\bar c}(\C), \tp(\bar c/\C))$ has joint.\ fin.\ EDRdeg with respect to each family $\Sigma$ of clopens in $S_{\bar c}(\C)$ of the form $\Sigma=\{[\varphi(\bar a,\bar y)]: \varphi(\bar x,\bar y)\}$ for some (finite) $\bar a$ (see Definition \ref{definition: joint. fin. EDERdeg with respect to Sigma}) is equivalent to the strengthening of Definition \ref{definition: EDEERP  and sep. fin. EDEERdeg from KLM}(2) where finite $\Delta$ is replaced by the set of all formulas in variables $\bar x$ corresponding to $\bar a$ and variables $\bar y$ corresponding to $\bar c$. In particular, Lemma \ref{lemma: main lemma for metrizability} specializes to a new result in the context of first order theories. The condition that $(\Aut(\C),S_{\bar c}(\C), \tp(\bar c/\C))$ has joint.\ fin.\ EDRdeg (see Definition \ref{definition: joint. fin. EDERdeg}) is equivalent to the further strengthening of  Definition \ref{definition: EDEERP  and sep. fin. EDEERdeg from KLM}(2) in which one requires the existence of a common $l$ good for all finite tuples $\bar a$ in $\C$ and where finite $\Delta$ is replaced by the set of all formulas in variables $\bar x$ corresponding to $\bar a$ and variables $\bar y$ corresponding to $\bar c$. As in the case of Lemma \ref{lemma: main lemma for metrizability}, Proposition \ref{proposition: joint fin EDERdeg} also specializes to a new result in the context of first order theories.

Bearing in mind Proposition \ref{proposition: specialization of EDERP to theories}, it is easy to see that Corollary \ref{cor: amenability_derp}(3) and Corollary \ref{corollary: characterization_sep_fin_EDERdeg} specialize to Theorems 3.14(ii) and 3.20 in \cite{KLM}, respectively. Similarly, using additionally Remark \ref{remark: DEERP iff WDEERP}, we see that  Corollary \ref{cor: amenability_derp}(1,2) specializes to Theorem 3.16(ii) in \cite{KLM}. In the other direction, Section 5 of \cite{KLM} yields big families of examples of theories with EDEERP or sep.\ fin.\ EDEERdeg which by Proposition \ref{proposition: specialization of EDERP to theories} automatically yield families of $0$-dimensional ambits with EDRP or sep.\ fin.\ EDRP. Most of these examples come from structural Ramsey theory. Namely, if a Fra\"{i}ss\'{e} class has the embedding Ramsey property [or finite Ramsey degree] and the Fra\"{i}ss\'{e} limit is $\aleph_0$-saturated, then the theory of the Fra\"{i}ss\'{e} limit has EDEERP [or sep.\ fin.\ EDEERdeg], but there are also many examples where the Fra\"{i}ss\'{e} class has only finite Ramsey degree but not the embedding Ramsey property, whereas the theory of the Fra\"{i}ss\'{e} limit has EDEERP. 

Example 5.7 in \cite{KLM} shows that DRP does not imply EDRP in the context of theories, let alone in general for $0$-dimensional ambits.
However, in the context of theories,  Itay Kaplan asked whether DEERP implies EDEERP under the NIP assumption. We now answer this question positively. 

\begin{proposition}\label{proposition: under NIP DERP iff EDERP for theories}
Assume $T$ has NIP. Then $T$ has DEERP iff $T$ has EDEERP.
\end{proposition}

\begin{proof}
Only ($\Rightarrow$) requires a proof.
Since $T$ has DEERP, by \cite[Theorem 3.16(ii)]{KLM}, the ambit $(\Aut(\C),S_{\bar c}(\C),\tp (\bar c/\C))$ has a fixed point (which means that $T$ is extremely amenable in the terminology from \cite{HKP}). Hence, by \cite[Theorem 7.7]{KNS}, the minimal left ideals in $E(S_{\bar c}(\C))$ are of bounded size. By \cite[Corollary 7.2]{KNS}, this implies that there exists $\eta \in E(S_{\bar c}(\C))$ whose image is contained in the set of types which do not fork over $\emptyset$. On the other hand, by \cite[Proposition 2.1]{HrPi}, a global type $p$ does not fork over $\emptyset$ if and only if $p$ is invariant under $\Autf_L(\C)$ (the group of Lascar strong automorphisms). By \cite[Proposition 4.2]{HKP}, extreme amenability of $T$ implies that $\Autf_L(\C)=\Aut(\C)$. Therefore, $\im(\eta) \subseteq \Inv(S_{\bar c}(\C))$. So $T$ has EDEERP by virtue of \cite[Theorem 3.14(ii)]{KLM}.
\end{proof}



The Boolean $\Aut(\C)$-algebra $\mathcal{A}$ associated with the $0$-dimensional ambit $(\Aut(\C),S_{\bar c}(\C),\tp(\bar c/\C))$ as described in Fact \ref{fact: 0-dim ambit is S(A)} consists of all subsets of $\Aut(\C)$ of the form 
$$\{\sigma \in \Aut(\C): \C \models \varphi(\sigma(\bar c), \bar c)\}$$ 
for a formula $\varphi(\bar x, \bar y)$. It is obvious that this algebra $\mathcal{A}$ is invariant also under the right translations by the elements of $\Aut(\C)$. Therefore, by Corollary \ref{corollary: right invariant G-algebras},

\begin{remark}\label{remark: DEERP iff WDEERP}
$(\Aut(\C),S_{\bar c}(\C),\tp(\bar c/\C))$ has WDRP if and only if it has DRP.
\end{remark}

Finally, we explain that Theorem 5.1 in \cite{MeSu} follows from our Corollary \ref{cor: amenability_derp}(3). Moreover, via this approach the assumption that the age of the ultrahomogeneous structure in \cite[Theorem 5.1]{MeSu} is countable becomes redundant. 

The context of \cite[Theorem 5.1]{MeSu} is the following. Let $M$ be an ultrahomogeneous structure. By ultrahomogeneity, for a finite tuple $\bar a$ in $M$, the set ${M\choose\bar a}$  of copies of $\bar a$ under the embeddings is the same as the copies of $\bar a$ under the automorphisms of $M$. The externally definable colorings of ${M\choose\bar a}$ are defined in the same way as in \cite{KLM}; see Definition \ref{definition: of definable colorings}(2) and Remark \ref{remark: ext def coloring}(1) (except the last sentence involving $\bar c$). 
By \cite[Definition 3.1 and 3.3]{MeSu}, $M$ is said to have {\em externally definable Ramsey property} if for every finite $\bar a\subseteq B\subseteq M$, and every externally definable coloring $c$ of ${M\choose\bar a}$ there is $B'\in{M\choose B}$ such that $|c[{B'\choose\bar a}]|=1$.
By \cite[Definition 4.2]{MeSu}, $M$ is said to have {\em fixed points on type spaces property (FPT)} if for every $n\geq 1$, every $\Aut(M)$-subflow of $S_n(M)$ has a fixed point. 

\begin{fact}\phantomsection\label{fact: Meur and Sullivan}\cite[Theorem 5.1]{MeSu}
\begin{enumerate}
\item If M has FPT, 
then $M$ has the externally definable Ramsey property.
\item If $\Age(M)$ is countable, then $M$ has the externally definable Ramsey property iff $M$ has FPT.
\end{enumerate}
\end{fact}

We generalize it to:

\begin{theorem}\label{theorem: improved Meir Sullivan}
$M$ has the externally definable Ramsey property iff $M$ has FPT.
\end{theorem}

\begin{proof}
($\Rightarrow$) Let $n \geq 1$. We want to show that every $\Aut(M)$-subflow $X$ of $S_n(M)$ has a fixed point. This is equivalent to showing that every minimal subflow $X$ is trivial. So consider $X$ minimal; it is an ambit with any point $p_0$ as the distinguished point with dense orbit. By the same argument as in Proposition \ref{proposition: specialization of EDERP to theories}(1), the externally definable Ramsey property of $M$ implies that $(\Aut(M),X,p_0)$ has EDRP. Thus, $X$ is trivial by Corollary \ref{cor: amenability_derp}(3).

($\Leftarrow$) 
Consider an externally definable coloring $c \colon {M\choose\bar a} \to 2^n$ of copies of some finite $\bar a$ defined via formulas $\varphi_0(\bar x, \bar y),\dots,\varphi_{n-1}(\bar x,\bar y)$ and types $p_0(\bar y)=\dots=p_{n-1}(\bar y) \in S_{\bar y}(M)$ (for a finite tuple $\bar y$). Let $X := \cl(\Aut(M) p_0)$, a subflow of $S_{\bar y}(M)$. By assumption, every minimal subflow of $X$ is trivial, so the $\Aut(M)$-ambit $(\Aut(M),X,p_0)$ has EDRP by  Corollary \ref{cor: amenability_derp}(3). Let $c' \colon \Aut(M) \to 2^n$ be the coloring given by $c'(\sigma)(i) := c(\sigma(\bar a))(i)$. It is clearly an externally definable (even strongly definable) coloring with respect to $(\Aut(M),X,p_0)$. Using EDRP as in the proof of Proposition \ref{proposition: specialization of EDERP to theories}(1), for every finite $B \subseteq M$ containing $\bar a$ we get a copy $B'$ of $B$ with $|c[{B'\choose\bar a}]|=1$.
\end{proof}

\subsection{Definable groups}\label{subsection: definable groups}

Let $G$ be a group definable in a first order structure $M$. There are two natural and very important $G$-flows considered in model theory: $(G,S_G(M))$ (that is, the flow of complete types over $M$ concentrated on $G$) and $(G,S_{G,ext}(M))$ (that is, the flow of complete external  types over $M$ concentrated on $G$, i.e.\ the space of ultrafilters on the Boolean algebra of externally definable subsets of $G$). In stable theories, they coincide, but in general not necessarily. Most of topological dynamics in model theory was done so far for the latter flow, because there is a left-continuous semigroup operation on $S_{G,ext}(M)$ extending the action of $G$ which makes it a semigroup isomorphic to $E(S_{G,ext}(M))$. In terms of Stone spaces, $S_G(M)$ can be identified with $S(\Def(G))$ (where $\Def(G)$ is the Boolean $G$-algebra of all definable subsets of $G$), whereas $S_{G,ext}(M):=S(\Extdef(G))$ (where $\Extdef(G)$ is the Boolean $G$-algebra of all externally definable subsets of $G$). There is an intermediate flow $(G,S(\Def(G)^d))$ which is isomorphic to $(G,E(S_G(M)))$ (see Fact \ref{fact: A^d}). In terms of Boolean $G$-algebras, we have $\Def(G) \subseteq \Def(G)^d \subseteq \Extdef(G)$.

Now, all the notions WDRP, DRP, EDRP (and, more generally, $S$-DRP for any $S$) can be applied to each of the above three flows. Similarly with the notions of sep.\ fin.\ DRdeg, sep.\ fin.\ EDRdeg, etc. While in the case of theories we were translating these notions to the corresponding Ramsey properties concerning colorings of copies of finite tuples of elements of the monster model (see Proposition \ref{proposition: specialization of EDERP to theories}), here the notions applied to the above three flows by definition concern colorings of translates of elements of $G$, so we do not need any further identifications. Then, Proposition \ref{prop: characteriazation_S-derp}, Corollary \ref{cor: amenability_derp}, Theorem \ref{theorem: characterization_sep_fin_S-DERdeg}, Corollary \ref{corollary: characterization_sep_fin_EDERdeg}, Proposition \ref{proposition: joint fin EDERdeg}, Lemma \ref{lemma: main lemma for metrizability}, and Corollary \ref{corollary: criterion for metrizability} are new ``KPT-style'' theorems for definable groups. 

Also, Proposition \ref{proposition: relationships} specializes to each of the above three flows. Applied to the simplest flow $(G,S_G(M))$ this yields a Ramsey-theoretic criterion (namely, sep.\ fin.\ EDRdeg for $(G,S_G(M))$) for profiniteness of the definable Bohr compactification of $G$ which by \cite[Proposition 3.4]{GPP} coincides with $\bar{G}/\bar{G}^{00}_M$ (equivalently, this is a criterion for the equality  $\bar{G}^0_M=\bar{G}^{00}_M$). Applied to  $(G,S_{G,ext}(M))$, it gives a criterion (i.e.\ sep.\ fin.\ EDRdeg for $(G,S_{G,ext}(M))$) for profiniteness of the externally definable Bohr compactification of $G$ whose definition is given right before \cite[Remark 1.16]{KrPi} and whose description in terms of model-theoretic connected components appears in \cite[Section 7]{KrPi}. (Here, $\bar{G}$ denotes the interpretation of $G$ in a monster model; $\bar{G}^{00}_M$ is the smallest $M$-type-definable subgroup of $\bar{G}$ of bounded index; $\bar{G}^{0}_M$ is the intersection of all $M$-definable subgroups of $\bar{G}$ of finite index; then $\bar{G}^{00}_M \leq \bar{G}^0_M$ are normal subgroups of $\bar{G}$, and $\bar{G}/\bar{G}^{00}_M$ and $\bar{G}/\bar{G}^{0}_M$ are compact groups (the latter one is even profinite) with the logic topology in which a subset is closed if its preimage in $\bar{G}$ under the quotient map is type-definable).

A fundamental question is how to find examples of definable groups with the various above properties. In the case of theories, one source of examples were stable theories, and another a variety of examples coming from structural Ramsey theory. For definable groups we still have the first class of examples, that is for every group $G$ definable in a stable structure $M$ the flow $(G,S_G(M))$ has sep.\ fin.\ EDRdeg, but we do not have a counterpart of structural Ramsey theory source. There are, however, ad hoc unstable examples such as $S_G(M)$ for $G:=(\mathbb{R},+)$ treated as a group definable in $M:=(\mathbb{R},+,\cdot)$, where EDRP holds. Let us discuss it a bit more carefully.

A correspondence between the dynamical property WAP and model-theoretic stability was first noted in \cite{Yaa} and \cite{YaTs}. Some particular forms of this phenomenon are included in the next fact. Recall that a formula $\varphi(x,y)$ is {\em stable} (in a theory $T$), if there is no $(a_i,b_i)_{i \in \omega}$ in a monster model of $T$ such that $\varphi(a_i,b_i) \iff i \leq j$; the formula $\varphi(x,y)$ is {\em stable in $M$} if is no such $(a_i,b_i)_{i \in \omega}$ in $M$.

\begin{fact}\phantomsection\label{fact: stability = WAP}
\begin{enumerate}
\item If $\Th(M)$ is stable, then $(G,S_G(M))$ is a WAP flow.
\item If $\varphi(y \cdot x)$ is stable in $M$ for every formula $\varphi(x)$ over $M$ implying $G(x)$, then $(G,S_{G}(M))$ is a WAP flow.
\item For any formula $\varphi(x)$ over $M$ implying $G(x)$: if $\varphi(y \cdot x)$ is stable in $M$, then $(G,S_{G,\varphi(y \cdot x)}(M))$ is a WAP flow (where $S_{G,\varphi(y \cdot x)}(M)$ denotes the space of all  complete $\varphi(y \cdot x)$-types over $M$ concentrated on $G$) .
\end{enumerate}
\end{fact}

\begin{proof}
(1) clearly follows from (2). Items (2) and (3) can be easily deduced from Grothendieck's double limit theorem as in \cite[Theorem 3.16]{HKP0}. Briefly, let $X=S_{G}(M)$ in (2) [and $X=S_{G,\varphi(y \cdot x)}(M)$ in (3)]. As in the proof of  \cite[Theorem 3.16]{HKP0}, by Grothendieck's theorem, we have that for any formula $\varphi(x)$ over $M$  implying $G(x)$, the characteristic function $\chi_{[\varphi(x)]}$ is WAP with respect to the flow $(G,X)$ if and only if the formula $\varphi(y \cdot x)$ is stable in $M$. So, by assumption, all functions $\chi_{[\varphi(x)]}$ are WAP in (2) [the function $\chi_{[\varphi(x)]}$ is WAP in (3)]. Bearing in mind that the flow $(G,X)$ is WAP iff all $f \in C(X)$ are WAP, the last conclusion, the fact that the WAP functions form a closed subalgebra of $C(X)$ in the supremum norm topology, the fact that all functions $\chi_{[\varphi(x)]}$ [the $G$-translates of the function $\chi_{[\varphi(x)]}$ in (3)] separate points in $X$, and Stone-Weierstrass theorem imply that $(G,X)$ is WAP.
\end{proof}

\begin{corollary}
In each item of Fact \ref{fact: stability = WAP}, the resulting $G$-ambit (that is: $(G,S_G(M),\tp(e/M))$ in (1) and (2), and $(G,S_{G,\varphi(y \cdot x)}(M),\tp_{\varphi(y \cdot x)}(e/M))$ in (3)) has sep.\ fin.\ EDRdeg.
\end{corollary}

\begin{proof}
It follows directly from Fact \ref{fact: stability = WAP} and Proposition \ref{proposition: WAP implies finite RP}.
\end{proof}

Let $\C \succ M$ be a monster model and $\bar{G} :=G(\bar \C)$. Recall that $G$ is called {\em definably amenable} if  there exists a finitely additive $G$-invariant probability measure on the Boolean algebra $\Def(G)$. 
The proof of the next lemma uses the idea of the proof of Proposition 7.11 from \cite{KNS}. For the notions of $f$-generic types and formulas and $G$-dividing formulas see \cite[Definition 3.2]{ChSi}.

\begin{lemma}\label{lemma: eta with f-generic image}
Assume that $\Th(M)$ has NIP and $G$ is definably amenable. Then there exists $\eta \in S_{G}(\C)$ whose image consists of f-generic types. 
\end{lemma}

\begin{proof}
For a formula $\varphi(x)$ and $p \in S_{G}(\C)$, set $X_{p,\varphi}:=\{\eta \in E(S_{G}(\C)): \neg \varphi(x) \in \eta(p)\}$. It is enough to show that 
$$\bigcap_{\varphi(x) \textrm{ non-f-generic}} \bigcap_{p \in S_{G}(\C)} X_{p,\varphi} \ne \emptyset.$$
So consider any formulas $\varphi_0(x),\dots, \varphi_{n-1}(x)$ which are not f-generic and any types $p_0,\dots,p_{n-1} \in S_{G}(\C)$. Then $\varphi(x):=\bigvee_{i=0}^{n-1} \varphi_i(x)$ is not f-generic by \cite[Corollary 3.5]{ChSi}. Let $M'\prec \C$ be a small model containing the parameters of $\varphi(x)$. By \cite[Proposition 3.4]{ChSi}, $\varphi(x)$ $G$-divides over $M'$, so there is an $M'$-indiscernible sequence $(g_i)_{i<\omega}$ such that the sequence $(g_i\varphi(x): i<\omega)$ is $k$-inconsistent for some $k<\omega$. By the pigeonhole principle, there is $0 \leq i\leq (k-1)n$ such that for all $j<n$ we have $g_i \varphi(x) \notin p_j$, equivalently $\neg \varphi(x) \in g_i^{-1} p_j$. So for $g :=g_i^{-1}$ we have $g \in \bigcap_{j<n}X_{p_j,\varphi_j}$. Thus, we have proved that the family $\{X_{p,\varphi}: \varphi(x) \textrm{ non-f-generic}, p\in S_{G}(\C)\}$ has the finite intersection property, which is enough by compactness of $E(S_{G}(\C))$ and closedness of the sets $X_{p,\varphi}$.
\end{proof}

\begin{remark}\phantomsection\label{remark: transfer of EDERP}
\begin{enumerate}
\item The property WDRP of $(G,S_G(M),\tp(e/M))$ is preserved under replacing $M$ by both an elementary extension or an elementary substructure (containing the parameters over which $G$ is defined). 
\item The property EDRP of $(G,S_G(M),\tp(e/M))$ is preserved under replacing $M$ by an elementary substructure.
\end{enumerate}
\end{remark}

\begin{proof}
(1) By Corollary \ref{cor: amenability_derp}, WDRP is equivalent to $\Inv(S_G(M)) \ne \emptyset$ (i.e.\ definable extreme amenability of $G$). It is well-known that this property is preserved under changing the model $M$. Going down is just by restriction of a $G$-invariant type to the smaller model. Going up is by the standard construction of extending a measure (here Dirac measure) to a bigger model.

(2) Consider any $M' \prec M$ containing the parameters over which $G$ is defined. The assumption that $(G,S_G(M),\tp(e/M))$ has EDRP is equivalent to the condition $(*)$ saying that for any formulas $\varphi_0(x,\bar a),\dots,\varphi_{n-1}(x,\bar a)$ over $M$ (i.e.\ $\bar a$ is contained in $M$), external elements $g_0,\dots,g_{n-1} \in \bar{G}$, and a finite subset $A \subseteq G=G(M)$, there exists $g \in G$ such that 
$$\bigwedge_{h,h' \in A}\bigwedge_{i<n} \varphi_i(h^{-1}g^{-1} g_i,\bar a) \leftrightarrow \varphi_i(h'^{-1}g^{-1} g_i,\bar a).$$

Our goal is to show that the counterpart of $(*)$ in $M'$ also holds. So consider any formulas $\varphi_0(x,\bar a),\dots,\varphi_{n-1}(x,\bar a)$ over $M'$, external elements $g_0,\dots,g_{n-1} \in \bar{G}$, and finite $A \subseteq G(M')$. Choose $g_0',\dots,g_{n-1}'$ so that $\tp(g_0',\dots, g_{n-1}'/M)$ is an heir extension of $\tp(g_0,\dots, g_{n-1}/M')$. Applying $(*)$ to the above data, we get $g \in G$ such that
$$\bigwedge_{h,h' \in A} \bigwedge_{i<n} \varphi_i(h^{-1}g^{-1} g_i',\bar a) \leftrightarrow \varphi_i(h'^{-1}g^{-1} g_i',\bar a).$$
Thus, by the definition of heir extensions, we can find $g' \in G(M')$ such that 
$$\bigwedge_{h,h' \in A} \bigwedge_{i<n} \varphi_i(h^{-1}g'^{-1} g_i',\bar a) \leftrightarrow \varphi_i(h'^{-1}g'^{-1} g_i',\bar a),$$
which completes the proof.
\end{proof}

The second item of the next corollary is a counterpart of Proposition \ref{proposition: under NIP DERP iff EDERP for theories} for definable groups.
Recall that, by Shelah's theorem (see \cite{She1}), assuming NIP, the component $\bar{G}^{00}_M$ is independent of $M$  so denoted by $\bar{G}^{00}$. Recall also that the {\em Shelah expansion} $M^{\textrm{Sh}}$ of $M$ is the expansion by adding as predicates all externally definable subsets of all finite Cartesian powers of $M$. Another Shelah's theorem (see \cite{She2}) says that if $M$ has NIP, then $M^{\textrm{Sh}}$ has quantifier elimination and NIP. Thus, under NIP, $(G,S_{G,ext}(M),\tp_{ext}(e/M))$ is isomorphic to $(G,S_G(M^{\textrm{Sh}}), \tp(e/M^{\textrm{Sh}}))$.

\begin{corollary}\phantomsection\label{corollary: DEERP iff EDERP under NIP for definable groups}
\begin{enumerate}
\item Assume that $\Th(M)$ has NIP and $G$ is definably amenable. Suppose that $\bar{G}^{00} = \bar{G}$. Then the ambit $(G,S_G(M),\tp(e/M))$ has EDRP. 
\item Assuming NIP, for the ambit $(G,S_G(M),\tp(e/M))$ the properties of having WDRP, DRP and EDRP are all equivalent.
\item Both items hold for $(G,S_{G,ext}(M),\tp_{ext}(e/M))$ in place of $(G,S_G(M),\tp(e/M))$.
\end{enumerate}
\end{corollary}

\begin{proof}
(1) By Lemma \ref{lemma: eta with f-generic image}, there is $\eta \in E(S_G(\C))$ whose image consists of f-generic types. By \cite[Proposition 3.8]{ChSi}, each global f-generic type is stabilized by $\bar{G}^{00}$ which, by our assumption, is the whole $\bar{G}$. Therefore, $\eta$ is $\bar{G}$-invariant, which by Corollary \ref{cor: amenability_derp}(3) is equivalent to EDRP for $(\bar{G},S_{G}(\C),\tp(e/\C))$.  Using Remark \ref{remark: transfer of EDERP}(2), we conclude that  $(G,S_{G}(M),\tp(e/M))$ has EDRP.

(2) By  Corollary \ref{cor: amenability_derp}, WDRP implies that $(G,S_G(M))$ has a fixed point, i.e.\ $G$ is definably extremely amenable. Therefore, $\bar{G}^{00}=\bar{G}$ (e.g., see \cite[Proposition 0.6]{KrPi-1}). So the conclusion follows from (1).

(3) follows from the above two items, the observation preceding Corollary \ref{corollary: DEERP iff EDERP under NIP for definable groups}, and the fact that each of the properties/objects: NIP, the component $\bar{G}^{00}$, and definable amenability of $G$ is invariant under passing to the Shelah expansion which we know by \cite{CPS}.
\end{proof}

\begin{corollary}\phantomsection\label{corollary: equivalences of RP for definable groups}
\begin{enumerate}
\item For the ambit $(G,S_{G,ext}(M),\tp_{ext}(e/M))$ all three properties WDRP, DRP, EDRP are equivalent. Similarly for the ambit $(G,S(\Def(G)^d),p_e)$.
\item For the ambit $(G,S_G(M),\tp(e/M))$ the properties WDRP and DRP are equivalent.
\item The ambit $(G,S_G(M),\tp(e/M))$ has EDRP iff  the ambit $(G,S(\Def(G)^d),p_e)$ has EDRP.
\end{enumerate}
\end{corollary}

\begin{proof}
(1) follows from Corollary \ref{corollary: equivalence of all RP for d-closed algebras} and Fact \ref{fact: A^d}, and (2) from Corollary \ref{corollary: right invariant G-algebras}. Item (3) follows from (1), Fact \ref{fact: A^d}, and Corollary \ref{cor: amenability_derp}.
\end{proof}

In the rest of this subsection, we give an example showing that DRP does not imply EDRP for $(G,S_G(M),\tp(e/M))$, and another example showing that EDRP for $(G,S_G(M),\tp(e/M))$ does not imply EDRP for $(G,S_{G,ext}(M),\tp_{ext}(e/M))$. The former example also shows that the externally definable Bohr compactification may differ from the definable Bohr compactification, confirming a prediction mentioned in the last paragraph of \cite{KrPi}.

The model completion from the first example below was mentioned to the first author by Ehud Hrushovski as a counterexample to a completely different statement concerning simple theories. 

\begin{example}\label{example: DERP but not EDERP fr definable groups}
Let $M=G$ be the Fra\"{i}ss\'{e} limit of the class of all finite groups of exponent $2$ with a predicate $D$ such that $D(0)$. (Equivalently, $M$ is the countable model of the model completion of the theory of groups of exponent $2$ with a predicate $D$ such that $D(0)$.)
\begin{enumerate}
\item $\bar{G}^{00}_M= \bar{G}$, whereas there exists an externally definable subgroup $H$ of $G$ of index $2$; even more, $H \in \Def(G)^d$.
\item $(G,S_G(M),\tp(e/M))$ has DRP but not EDRP.
\item The definable Bohr compactification of $G$ is trivial, whereas the externally definable Bohr compactification of $G$ is non-trivial. 
\end{enumerate}
\end{example}

\begin{proof}
(1) We start from
\begin{clm}
$G$ does not have any proper, infinite definable subgroup. 
\end{clm}

\begin{clmproof}
Suppose $H$ is such a subgroup. By quantifier elimination, $H$ is the set of realizations of a formula 
$$\varphi(x)=\left(\bigvee_{a \in A} x=a\right) \vee \bigvee_{i <n} \left(\bigwedge_{j < j_i} x \ne a_{ij} \wedge   \bigwedge_{k <k_i}\bigwedge_{l <l_i} D(x-b_{ik}) \wedge \neg D(x-c_{il})\right),$$
for some finite subset $A$ of $G$, $n>0$, $j_i,k_i,l_i \geq 0$, and some elements $a_{ij},b_{ik},c_{il} \in G$. Since the complement of $H$ in $G$ is infinite, there is $g \in G$ such that for every $i$ either there is $k_i'<k_i$ with $\neg D(g-b_{ik_i'})$ or there is $l_i'<l_i$ with $D(g-c_{il_i'})$. For each $i <n$, let
$$\psi_i(x):=
\begin{cases}
\neg D(x-b_{ik_i'}) & \mbox{ if $k_i'$ exists},\\
D(x-c_{il_i'}) & \mbox{ if $k_i'$ does not exist (which implies that $l_i'$ exists)}.
\end{cases}$$
The element $g$ witnesses that $\bigwedge_{i<n} \psi_i(x)$ is consistent, which implies that 
$$\{b_{ik_i'}: i<n \textrm{ such that } k_i' \textrm{ exists}\} \cap \{c_{il_i'}:  i<n \textrm{ such that } k_i' \textrm{ does not exist}\} = \emptyset.$$

Since $H$ is infinite, there is $i_0<n$ such that 
$$\varphi_{i_0}(x):=\bigwedge_{k <k_{i_0}}\bigwedge_{l <l_{i_0}} D(x-b_{i_0k}) \wedge \neg D(x-c_{i_0l})$$
is consistent, equivalently $\{b_{i_0k}: k<k_{i_0}\} \cap \{c_{i_0l}:  l<l_{i_0}\} = \emptyset$.
As $H$ is infinite, there is $h_1 \in H \setminus \langle \{b_{ik},c_{il}: i<n,k<k_i,l<l_i\}\rangle$. 

The above context allows us to find $h_2 \in \varphi_{i_0}(G) \setminus \{a_{i_0 j}: j<j_{i_0}\}$ (and so $h_2 \in H$) such that $h_2 + h_1 \notin A$ and $h_2 +h_1 \in \bigcap_{i<n} \psi_i(G)$ (and so $h_2+h_1 \notin H$). Thus, $H$ is not closed under addition, a contradiction.
\end{clmproof}

Since $\bar{G}/\bar{G}^{00}_M$ is a compact group of exponent $2$, it must be profinite (e.g., see \cite[Theorem 28.20]{HeRo}). Thus, $\bar{G}^{00}_M=\bar{G}^0_M$ which is equal to $\bar{G}$ by Claim 1. 

Let now $H$ be any subgroup of $G$ of index $2$. By compactness, there is $g \in \bar{G}$ such that  for every $a \in G$ we have
$$\models D(g-a) \Leftrightarrow a \in H,$$
which shows that $H$ is externally definable. To get the stronger conclusion that $H \in \Def(G)^d$, take $p \in S_G(M)$ determined by 
$$D(x-a) \in p \Leftrightarrow a \in H.$$
Then $d_p(D):=\{g \in G: D \in g+p\} =\{g\in G: D(x+g) \in p\} =-H=H$, so $H \in \Def(G)^d$.

(2) The non-algebraic type $p \in S_G(M)$ containing $D(x-g)$ for all $g \in G$ is clearly $G$-invariant, so $(G,S_G(M),\tp(e/M))$ has DRP by Corollaries \ref{cor: amenability_derp}(1) and \ref{corollary: equivalences of RP for definable groups}(2).

By Corollary \ref{cor: amenability_derp}(3), the property EDRP for $(G,S_G(M),\tp(e/M))$ is equivalent to the existence of a $G$-invariant element in $E(S_{G}(M))$. By Fact \ref{fact: A^d}, the $G$-flow $E(S_G(M))$ is isomorphic to the $G$-flow $S(\Def(G)^d)$. But the finite index subgroup $H \in \Def(G)^d$ (which we have by item (1)) guarantees that there is no $G$-invariant ultrafilter in  $S(\Def(G)^d)$. All of this implies that $(G,S_G(M),\tp(e/M))$ does not have EDRP.

(3) Since the definable Bohr compactification of $G$ is the quotient map $G \to \bar{G}/\bar{G}^{00}_M$, item (1) shows that it is trivial. On the other hand, taking an externally definable subgroup $H$ of index $2$ provided by item (1), the quotient map $G \to G/H$ is a non-trivial externally definable compactification of $G$.
\end{proof}

\begin{example}
Let $(G,+)$ be an infinite group of exponent $2$, so naturally a vector space over $\mathbb{F}_2$. Let $P:=\mathcal{P}_{\textrm{fin}}(G)$ be the family of all finite subsets of $G$. We define a $2$-sorted structure $M$ as follows. The sorts are $(G,+)$ and $P$, and we add a predicate $R\subseteq G \times P$ defined by 
$$R(g,s) \Leftrightarrow g \in s.$$ 
Then $(G,S_G(M),\tp(e/M))$ has EDRP, whereas $(G,S_{G,ext}(M),\tp_{ext}(e/M))$ does not have EDRP.
\end{example}

\begin{proof}
First, note that each definable subset of the sort $G$ is finite or co-finite. Indeed, consider any finite set $A=A_1 \cup A_2$  of parameters from $M$, where $A_1 \subseteq G$ and $A_2 \subseteq P$. Let $A'$ be the subgroup of $G$ generated by $A_1 \cup \bigcup A_2$; it is a finite subgroup (so subspace) of $G$. It is clear that $\Aut(G/A')$ acts transitively on $G \setminus A'$. Since each subset of $G$ definable over $A$ is a union of orbits over $A$ and so also a union of orbits over $A'$, we get that it is finite (if it is contained in $A'$) or co-finite (if it meets $G \setminus A'$).

Therefore, $S_G(M)$ is the one-point compactification of the discrete space $G$ (we identify the realized types with the elements of $G$); denote the non-realized type by $\infty$.

Let $(g_i)$ be a net of elements of $G$ eventually outside every finite subset of $G$. Then $\lim g_i=\infty$ in $S_G(M)$. Now, let $\eta \in E(S_G(M))$ be the limit of the net $(g_i)$, where the $g_i$'s are treated as translates $\pi_{g_i}(p) := g_i +p$. Then $\eta$ is the constant function taking value $\infty$ which is clearly invariant under the action of $G$. Thus, $(G,S_G(M),\tp(e/M))$ has EDRP by Corollary \ref{cor: amenability_derp}.

On the other hand, by compactness, we easily get that every subset of $G$ is externally definable, so $S_{G,ext}(M) = \beta G$. 
Hence, any subgroup $H$ of $G$ of index $2$ guarantees that there is no $G$-invariant type $p \in S_{G,ext}(M)$. Thus, by Corollary \ref{cor: amenability_derp}, $(G,S_{G,ext}(M),\tp_{ext}(e/M))$ does not have WDRP, so DRP and EDRP either.
\end{proof}

The theory of $M$ in the above example has IP. Paper \cite{CPS} shows that many properties (e.g.\ definable amenability) are preserved under passing to the Shelah expansion. This naturally leads to the following question.

\begin{question}
Is it true that under NIP, 
$$(G,S_G(M),\tp(e/M)) \textrm{ has EDRP iff } (G,S_{G,ext}(M),\tp_{ext}(e/M)) \textrm{ has EDRP}?$$ 
By virtue of Corollaries \ref{corollary: DEERP iff EDERP under NIP for definable groups} and \ref{cor: amenability_derp}, the question is equivalent to: Is definable extreme amenability of $G$ preserved under passing to the Shelah expansion?
\end{question}

\subsection{KPT context}\label{subsection: KPT context}

Here, we apply the abstract context from Section \ref{section: main} to the universal ambits of groups of automorphisms of first order structures treated as topological groups with the pointwise convergence topology.

First, we recall a useful model-theoretic description of the universal ambits of such groups of automorphisms from \cite[Section 2]{KrPi2}. Let $M$ be a first order structure and $G:=\Aut(M)$ be its group of automorphisms equipped with the pointwise convergence topology. Let $\mathcal{M}$ be the structure consisting of two disjoint sorts $G$ and $M$ with predicates for all subsets of all finite Cartesian products of sorts; we call this language {\em full}. Note that the natural action of $G$ on $M$ is $\emptyset$-definable in $\mathcal{M}$, all elements of $\mathcal{M}$ are in $\dcl(\emptyset)$, and all $L$-definable subsets of the finite Cartesian powers of $M$ are $\emptyset$-definable in $\mathcal{M}$. Hence, $L$-formulas can naturally be identified with equivalent formulas from the full language. Types in this full language will be denoted by $\tp^{\textrm{full}}$. 
Let $\mathcal{M}^*=(G^*,M^*,\dots) \succ \mathcal{M}$ be a monster model (of the theory of $\mathcal{M}$). Then $G^*$ acts definably and faithfully as a group of automorphisms of $M^*$ treated as an $L$-structure. Enumerate $M$ as $\bar m$. Define
$$\Sigma^{\mathcal{M}} := \{\tp^{\textrm{full}} (\sigma(\bar m)): \sigma \in G^*\} = \{\tp^{\textrm{full}} (\sigma(\bar m)/\mathcal{M}): \sigma \in G^*\} .$$
%

\begin{fact}[{\cite[Theorem 2.2]{KrPi2}}]\label{fact: universal G-ambit}
The formula $g \cdot \tp^{\textrm{full}}(\sigma(\bar m)) := \tp^{\textrm{full}}(\sigma(g^{-1}(\bar m)))$ yields a well-defined left action of $G$ on $\Sigma^{\mathcal{M}}$, and with this action $(G,\Sigma^{\mathcal{M}}, \tp^{\textrm{full}}(\bar m))$ is the universal left $G$-ambit. In particular, this universal ambit is zero-dimensional.
\end{fact}

Let us now recall some basic notions from structural Ramsey theory. It is usually done for Fra\"{i}ss\'{e} classes and Fra\"{i}ss\'{e} structures, but let us look at a more general context of an arbitrary structure $M$ as in \cite{KrPi2}. 

For any tuple $\bar a$ in $M$ and $B \subseteq M$, by ${B \choose \bar a}$ we will mean the set $\{ \bar a' : \bar a' \subseteq B \;\, \mbox{and}\;\, \bar a'= f(\bar a)\;\, \mbox{for some}\;\, f \in \Aut(M)\}$; an analogous notation applies when $\bar a$ is replaced by a subset $A$ of $M$.

\begin{definition}\phantomsection\label{definition: Ramsey property}
\begin{enumerate}
\item We say that $M$ has the {\em embedding Ramsey property} (ERP) if for any finite tuple $\bar a$ in $M$ and a finite set $B\subseteq M$ containing $\bar a$, for any $r \in \omega$, there is a finite $C \subseteq M$ containing $B$ such that for every coloring $c\colon {C \choose \bar a} \to r$ there is $B' \in {C \choose B}$ such that ${B' \choose \bar a}$ is monochromatic with respect to $c$.
\item We say that $M$ has {\em separately finite embedding Ramsey degree} (sep.\ fin.\ ERdeg) if for any finite tuple $\bar a$ there exists $k_{\bar a} \in \omega$ such that for every finite $B \subseteq M$ containing $\bar a$ and for any $r \in \omega$ there is a finite $C\subseteq M$ containing $B$ such that for every coloring $c\colon {C \choose \bar a} \to r$ there is $B' \in {C \choose B}$ such that the set $c[{B' \choose \bar a}]$ is of size at most $k_{\bar a}$.
\end{enumerate}
\end{definition}
In the above definition, it is equivalent to replace a finite $C$ by the whole $M$ (see Remarks 3.1 and 6.1 in \cite{KrPi2}).

A fundamental result of KPT theory \cite{KPT} (see \cite[Theorem 3.2]{KrPi2} for arbitrary $M$) says that $M$ has ERP if and only if $G=\Aut(M)$ is extremely amenable (that is, $\Sigma^{\mathcal{M}}$ has a fixed point). Another fundamental result obtained in \cite[Theorem 8.7]{Zuc} (see also \cite[Theorem 6.3]{KrPi2} for an alternative proof) says that a countable $M$ has sep.\ fin.\ ERdeg if and only if the universal minimal $G$-flow is metrizable.

\begin{remark}\label{remark: all RP agree in Sigma^M}
All three properties WDRP, DRP, EDRP of the ambit $(G,\Sigma^{\mathcal{M}},\tp^{\textrm{full}}(\bar m))$ are equivalent. They are also equivalent to $M$ having the ERP.
\end{remark}

\begin{proof}
The first part follows from Corollary \ref{corollary: equivalence of all RP for d-closed algebras}, because on the universal $G$-ambit there is a left continuous semigroup operation extending the action of $G$. The second part follows immediately from the same dynamical characterizations of WDRP of $(G,\Sigma^{\mathcal{M}},\tp^{\textrm{full}}(\bar m))$ and of ERP of $M$, namely the existence of a fixed point in $\Sigma^{\mathcal{M}}$, provided by Corollary \ref{cor: amenability_derp} and \cite[Theorem 3.2]{KrPi2}.
\end{proof}

It is more insightful to see the above remark more directly in terms of colorings. This approach will be used to compare the versions of the notion of sep.\ fin.\ Ramsey degree.

\begin{remark}\label{remark: the notions of colorings are the same for Sigma^M}
The notions of strongly definable, definable, and externally definable coloring coincide for the ambit $\Sigma^{\mathcal{M}}$. This implies the first sentence of Remark \ref{remark: all RP agree in Sigma^M}.
\end{remark}

\begin{proof}
An externally definable coloring $c \colon G \to 2^n$ with respect to $\Sigma^{\mathcal{M}}$ is given by
$$c(g)(i)= \left\{
\begin{array}{cl}
1, & \models \varphi_i(g_i g (\bar m))\\
0, & \models \lnot\varphi_i(g_i g (\bar m))
\end{array}
\right.$$
for some formulas $\varphi_i(\bar x)$ in the full language and some $g_i \in G^*$. Although $g_i$ is an external parameter, the set of realizations of the formula $\varphi_i(g_i (\bar x))$ in $M$ is definable in $M$ by a formula $\varphi_i'(\bar x)$ in the full language. Then
$$c(g)(i)= \left\{
\begin{array}{cl}
1, & \models \varphi_i'(g (\bar m))\\
0, & \models \lnot\varphi_i'(g (\bar m))
\end{array}
\right.,$$
so $c$ is strongly definable.
\end{proof}

As in the context of theories from Subsection \ref{subsection: first order theories}, there is a correspondence between arbitrary colorings of the $G$-copies of a finite tuple from $M$ and strongly definable colorings of elements of $G$ with respect to the $G$-ambit $\Sigma^{\mathcal{M}}$. 

Namely, for a given finite tuple $\bar a$ from $M$ (where without loss of generality the coordinates of $\bar a$ are pairwise distinct) consider a coloring  $c \colon {M \choose \bar a} \to n$. Then, for every $i<n$,  $c^{-1}(i) = \varphi_i(M)$ for some formula $\varphi_i(\bar x')$ in the full language. Let $\bar x \supseteq \bar x'$ be a tuple of variables corresponding to $\bar m$. We obtain a strongly definable coloring $c' \colon G \to 2^n$ given by
$$c'(g)(i)= \left\{
\begin{array}{cll}
1, & \models \varphi_i(g (\bar a)) & (\Leftrightarrow \tp(\bar m) \in g[\varphi_i(\bar x)])\\
0, & \models \lnot\varphi_i(g (\bar a)) & (\Leftrightarrow \tp(\bar m) \notin g[\varphi_i(\bar x)])
\end{array}
\right..$$  

Conversely, start from a strongly definable coloring $c' \colon G \to 2^n$ given by 
$$c'(g)(i)= \left\{
\begin{array}{cll}
1, & \models \tp(\bar m) \in g[\varphi_i(\bar x)]\\
0, & \models \tp(\bar m) \notin g[\varphi_i(\bar x)]
\end{array}
\right.$$  
for some formulas $\varphi_i(\bar x)$. Let $\bar a$ be a finite tuple in $M$ corresponding to the finitely many variables used in $\varphi_0(\bar x),\dots,\varphi_{n-1}(\bar x)$. This  yields a coloring $c \colon  {M \choose \bar a} \to 2^n$ given by 
$$c(g(\bar a))(i)= \left\{
\begin{array}{cll}
1, & \models \varphi_i(g (\bar a))\\
0, & \models \lnot\varphi_i(g (\bar a)) 
\end{array}
\right..$$  

From these translations, the second part of Remark \ref{remark: all RP agree in Sigma^M} follows easily.

While in the context of $(G,\Sigma^{\mathcal{M}},\tp^{\textrm{full}}(\bar m))$ all the Ramsey properties agree and Corollary \ref{cor: amenability_derp} recovers the fundamental theorem of KPT theory, the situation for Ramsey degrees seems different regarding the notion of sep.\ fin.\  EDRdeg. But first we briefly discuss joint.\ fin.\ EDRdeg.

Recall that the notion of joint.\ fin.\ EDRdeg with respect to a family $\Sigma$ of clopen subsets of an ambit was introduced in Definition \ref{definition: joint. fin. EDERdeg with respect to Sigma}. The notions of joint.\ fin.\ DRdeg and joint.\ fin.\ WDRdeg with respect to $\Sigma$ are defined analogously (using definable and strongly definable $\Sigma$-colorings, respectively).

\begin{remark}\label{remark: sep. fin. ERdeg vs joint. fin. EDErdeg}
Let $\bar{x}'$ be a finite subtuple of the infinite tuple $\bar x$ of variables corresponding to $\bar m$. Let $\Sigma_{\bar{x}'}$ be the collection of all clopens in $\Sigma^{\mathcal{M}}$ of the form $[\varphi(\bar{x}')]$, where $\varphi(\bar{x}')$ ranges over the formulas in the full language in variables $\bar{x}'$. Then joint.\ fin.\ EDRdeg, joint.\ fin.\ DRdeg, and joint.\ fin.\ WDRdeg, all with respect to $\Sigma_{\bar{x}'}$, are all equivalent. And the property that these equivalent conditions hold for all finite $\bar{x}' \subseteq \bar x$ is equivalent to $M$ having sep.\ fin.\ ERdeg.
\end{remark}

\begin{proof}
The equivalence of the three notions follows via an obvious analog of Remark \ref{remark: the notions of colorings are the same for Sigma^M} for $\Sigma_{\bar{x}'}$-colorings. The last part follows from the translations explained after Remark \ref{remark: the notions of colorings are the same for Sigma^M}.
\end{proof}

Using Remark \ref{remark: sep. fin. ERdeg vs joint. fin. EDErdeg}, observe that, for a countable structure $M$, Corollary \ref{corollary: criterion for metrizability} applied to the countable cofinal, upward directed family $\mathcal{D}:=\{\Sigma_{\bar{x}'}: \bar{x}' \subseteq \bar{x} \textrm{ finite}\}$ of collections of clopens in $\Sigma^{\mathcal{M}}$ implies the aforementioned Zucker's theorem (see \cite[Theorem 8.7]{Zuc} and \cite[Theorem 6.3]{KrPi2}) saying that if $M$ has sep.\ fin.\ ERdeg, then the universal minimal $\Aut(M)$-flow is metrizable.

The main new notion in Section \ref{section: main} is that of sep.\ fin.\ EDRdeg with a dynamical characterization provided by Corollary \ref{corollary: characterization_sep_fin_EDERdeg} and main consequences obtained in Corollaries \ref{corollary: Ellis group is profinite}, \ref{corollary: sep. fin. EDERdeg implies profiniteness of distal factors}, and Proposition \ref{proposition: relationships}. This notion applied to the ambit $(G,\Sigma^{\mathcal{M}}, \tp^{\textrm{full}}(\bar m))$ seems strictly weaker than $M$ having sep.\ fin.\ ERdeg, but we do not know an example showing strictness. The point is that by Remark \ref{remark: sep. fin. ERdeg vs joint. fin. EDErdeg}, sep.\ fin.\ ERdeg for $M$ is equivalent to joint.\ fin.\ EDRdeg with respect to every $\Sigma_{\bar{x}'}$ (with $\bar{x}' \subseteq \bar{x}$ finite), whereas sep.\ fin.\ EDRdeg means joint.\ fin.\ EDRdeg with respect to every finite subfamily of every $\Sigma_{\bar{x}'}$. 

\begin{problem}\label{problem: sep. fin EDERdeg and sep. fin ERdeg}
Find an example (or prove there is no one) showing that sep.\ fin.\ EDRdeg for $(G,\Sigma^{\mathcal{M}}, \tp^{\textrm{full}}(\bar m))$ does not imply that $M$ has sep.\ fin.\ ERdeg.
\end{problem}

Corollaries \ref{corollary: Ellis group is profinite} and \ref{corollary: sep. fin. EDERdeg implies profiniteness of distal factors} specialize to the following result.

\begin{corollary}\label{corollary: main new result in KPT context}
If $(\Aut(M),\Sigma^{\mathcal{M}}, \tp^{\textrm{full}}(\bar m))$ has sep.\ fin.\ EDRdeg, then the Ellis group of this ambit is profinite and every distal minimal  $\Aut(M)$-flow is a profinite flow (i.e.\ an inverse limit of finite flows). In particular, if $M$ has sep.\ fin.\ ERdeg (which for countable $M$ is equivalent to metrizability of the universal minimal $\Aut(M)$-flow), then every distal minimal $\Aut(M)$-flow is a profinite flow.
\end{corollary}

Regarding examples, for every structure $M$ from the huge family of Fra\"{i}ss\'{e} structures with finite ERdeg we clearly have that $(G,\Sigma^{\mathcal{M}}, \tp^{\textrm{full}}(\bar m))$ has sep.\ fin.\ EDRdeg. A simple non-example is easily obtained from Example \ref{example: not sep. fin. DERdeg}:

\begin{example}
For $M:=(\mathbb{Z},\leq)$ the ambit $(G,\Sigma^{\mathcal{M}}, \tp^{\textrm{full}}(\bar m))$ does not have sep.\ fin.\ EDRdeg (it does not even have sep.\ fin.\ DRdeg).
\end{example}

\begin{proof}
$\Aut(M)$ consists of integer shifts, so is isomorphic to $(\mathbb{Z},+)$. Since it is countable and Polish (as a closed subgroup of $S_\infty$), it must be discrete. Thus $\Sigma^{\mathcal{M}} = \beta G$, and the conclusion follows from Example \ref{example: not sep. fin. DERdeg}.
\end{proof}

\section{Comments, questions, and further examples}\label{section: questions and examples}

Regarding the conditions in Proposition \ref{proposition: relationships}, in the context of first order theories from \cite{KLM} (see also Subsection \ref{subsection: first order theories}), we do not know any examples which would show that $(A) \not\Rightarrow (B)$ and that $(C) \not\Rightarrow (D)$; similarly in the context of definable groups from Subsection \ref{subsection: definable groups} and in the classical KPT context (discussed in Subsection \ref{subsection: KPT context}). Regarding  $(C) \Rightarrow (D)$, we do not know any counterexample even in the abstract context.

\begin{question}
Does $(C)$ imply $(D)$ in Proposition \ref{proposition: relationships}.
\end{question}

However, in the general abstract context from Proposition \ref{proposition: relationships}, we give an example showing that $(A) \not\Rightarrow (B)$. Before that let us briefly describe a background concerning Toeplitz flows. The relevant notions and facts can be found in \cite{Dow} or in the preliminaries in \cite{SeSi} and in the references there.

Let $s=(s_i)$ be a sequence of positive integers such that $s_{i+1} \mid s_i$ for all $i$. The {\em $s$-odometer} is the profinite space of $s$-adic integers $\invlim_{i} \mathbb{Z}/s_i\mathbb{Z}$ with the action of $\mathbb{Z}$ given by $n \cdot x:= n+x$. By a {\em 2-adic odometer} we mean the $s$-odometer for $s_i=2^i$, i.e.\ $Z_2:= \invlim_i \mathbb{Z}/2^i \mathbb{Z}$. Every odometer is a minimal equicontinuous flow.

Let $A$ be a finite set (in our example, $A:=\{0,1\}$ is enough). A {\em subshift} over $A$ is any $\mathbb{Z}$-subflow of the shift $\mathbb{Z}$-flow $A^{\mathbb{Z}}$. It is easy to see that an infinite subshift is never equicontinuous. A {\em Toeplitz sequence} is a sequence $\omega \colon \mathbb{I} \to A$ (where $\mathbb{I}= \mathbb{N}$ or $\mathbb{I}= \mathbb{Z}$) satisfying: for every $n \in \mathbb{I}$ there is $l \in \mathbb{N}$ such that for every $k \in \mathbb{N}$ we have $\omega(n)=\omega(n + kl)$. A {\em Toeplitz flow} is the subflow of $A^I$ generated by a Toeplitz sequence. It turns out that each Toeplitz flow is minimal  and is an almost 1-1 extension of an odometer which is its maximal equicontinuous factor (where recall that $X$ is an {\em almost 1-1} extension of $Y$ if there an epimorphism $\pi \colon X \to Y$ such that the set $\{ y \in Y: |\pi^{-1}(y)|=1\}$ is dense in $Y$). If $Z$ is a factor of a Toeplitz flow $X$, then $Z$ is an almost 1-1 extension of a factor of the maximal equicontinuous factor of $X$. A Toeplitz flow will be  called {\em proper} if it is not equicontinuous. Since infinite subshifts are never equicontinuous, each Toeplitz flow generated by an aperiodic Toeplitz sequence is proper. An example of an aperiodic Toeplitz sequence generating a Toeplitz flow with maximal equicontinuous factor being the $2$-adic odometer $Z_2$ can be found in \cite[Section 10]{Dow}; the obtained Toeplitz flow is therefore proper.

\begin{example}\label{example: (A) does not imply (B)}
Let $X$ be a proper Toeplitz flow which is an almost 1-1 extension of the odometer $Z_2:=\invlim_i \mathbb{Z}/2^i \mathbb{Z}$ (naturally treated as a $\mathbb{Z}$-flow). Then the Ellis group of $(\mathbb{Z},X)$ is the group $Z_2$ of $2$-adic integers, whereas $X$ is not isomorphic to an inverse limit $\invlim_j X_i$ of $\mathbb{Z}$-flows $X_i$ with finite Ellis groups. 
\end{example}

\begin{proof}
Since the natural embedding of $\mathbb{Z}$ into $Z_2$ is a group compactification of $\mathbb{Z}$, we get that the Ellis semigroup $E(Z_2)$ is a compact group topologically isomorphic to the profinite group $Z_2$. Since  $E(Z_2)$ is a compact group, it coincides (as a topological group) with the Ellis group of the odometer $Z_2$ (e.g.\ by Lemmas 3.11 and 3.12 in the first arXiv version of \cite{KLM}). Thus, the Ellis group of the odometer $Z_2$ is profinite. Therefore, profiniteness of the Ellis group of the flow $(\mathbb{Z},X)$ follows from \cite[Lemma 5.7]{CGK}, which tells us that the Ellis group of a minimal flow is preserved under taking minimal almost 1-1 extensions.

Now, suppose for a contradiction that $X \cong \invlim_j X_j$, where each $X_j$ has finite Ellis group. Since each $X_j$ is factor of $X$, it is an almost 1-1 extension of a factor $Y_j$ of $Z_2$. 
Since the Ellis group of $Y_j$ is finite and $Y_j$ is equicontinuous and minimal, we deduce that the Ellis semigroup $E(Y_j)$ is finite and so is $Y_j$ as a quotient of $E(Y_j)$. Hence, as $(\mathbb{Z},Y_j)$ is minimal, we conclude the action of $\mathbb{Z}$ on $Y_j$ is transitive. Therefore, since $X_j$ is an almost 1-1 extension of $Y_j$, we get that it is a 1-1 extension, so $X_j$ is finite. Thus, $X$ is a profinite flow, and so it is equicontinuous. This contradicts the assumption that $X$ was a proper Toeplitz flow.
\end{proof}

Proposition \ref{proposition: relationships} yields a Ramsey-theoretic criterion (namely sep.\ fin.\ EDRdeg) for profiniteness of the Ellis group of $X$ and in consequence for profiniteness of the Bohr compactification below $(G,X,x_0)$. But it is not a characterization (i.e.\ an equivalent condition) which is illustrated by Example \ref{example: (A) does not imply (B)}. It would be interesting to find Ramsey-theoretic characterizations of profiniteness of the Ellis group of $(G,X)$ and also of the Bohr compactification below $(G,X,x_0)$. This would be particularity important in the context of first order theories, where (as will be shown in a forthcoming paper of the first author with Daniel Hoffmann) the Bohr compactification below $(\Aut(\C),S_{\bar c}(\C),\tp(\bar c /\C))$ is precisely $\Gal_{KP}(T)$. It would be a combinatorial characterization of equality of Shelah and Kim-Pillay strong types, which could shed new light on a long-standing conjecture that this equality holds in simple theories. Regarding first order theories, we do not know whether sep.\ fin.\ EDRdeg is a characterization of profiniteness of the Ellis group or only a criterion, as we do not know whether $(A)$ implies $(B)$ and whether $(C)$ implies $(D)$. But we know that profiniteness of $\Gal_{KP}(T)$ is strictly weaker than profiniteness of the Ellis group, as is shown by Example 5.11 in \cite{KLM}.

Regarding the relationships  between EDRP, DRP, WDRP, by Example \ref{example: WDERP does not imply DERP},  WDRP $\not\Rightarrow$ DRP, although both these notions are equivalent in all three specializations considered in this paper. The implication DRP $\Rightarrow$ EDRP does not hold for theories and for definable groups (by \cite[Example 5.7]{KLM} and Example \ref{example: DERP but not EDERP fr definable groups}), but it holds in the KPT context. However, assuming NIP, it holds for theories (by Proposition \ref{proposition: under NIP DERP iff EDERP for theories}) and also for definable groups (by Corollary \ref{corollary: DEERP iff EDERP under NIP for definable groups}). 
It is well-known (first observed independently by Chernikov and Simon and by Ibarlucia) that NIP corresponds to tameness in topological dynamics. A precise statement of this form (see e.g.\ \cite[Corollary 5.8]{KrRz}) is that a theory $T$ has NIP if and only if the flow $(\Aut(\C), S_{\bar c}(\C))$ is tame. So this suggests a question whether  tameness is a sufficient assumption for the implication DRP $\Rightarrow$ EDRP.
It turns out that there exists a counterexample 
(which we skip, as it has been recently found with a help of AI during the work on \cite{HoKr})
 In Corollary \ref{cor: amenability_derp} and Remark \ref{remark: fixed points in powers}, we got the following dynamical characterizations of EDRP of an arbitrary $0$-dimensional ambit $(G,X,x_0)$:

\begin{enumerate}
\item Every subflow of $X$ has a fixed point (which we could call {\em hereditary extreme amenability} of (G,X)).
\item Every subflow of every finite Cartesian power $X^n$ has a fixed point (i.e., $X^n$ is hereditarily extremely amenable).
\item $E(X)$ has a fixed point (i.e., $E(X)$ is extremely amenable).
\end{enumerate} 

While these three conditions are equivalent without the tameness assumption, it is interesting to remark that, by the very recent preprint \cite{HoKr} (whose results were obtained much later than all the results of this paper), when extreme amenability is replaced by amenability, then the last two items above remain equivalent in general but the first two do not, and one of the main results in \cite{HoKr} says that they are equivalent under tameness. 

By Remark \ref{remark: equivalent dynamical characterizations of EDERP} and Corollary \ref{cor: amenability_derp}, EDRP is a Ramsey-theoretic characterization of triviality of the minimal left ideals in $E(X)$, and so it is a criterion for triviality of the Ellis group of $(G,X)$ and of the Bohr compactification below $(G,X,x_0)$. It would be interesting to find characterizations of each of the last two properties. In \cite{The}, there is a Ramsey-theoretic characterization of triviality of the Bohr compactification in the KPT context. This kind of result for first order theories could be very interesting, as it would be a Ramsey-theoretic characterization of equality of Kim-Pillay types and complete types. It is not hard to give strongly related examples in the context of first order theories and in the KPT context showing that triviality of the Ellis group does not not even imply DRP. Namely, in the context of theories, take $T$ to be the theory of the unit complex circle equipped with the ternary relation of the circular order. $T$ has quantifier elimination, NIP, and $\emptyset$ is not an extension base for forking. Thus, $T$ is not amenable by \cite[Corollary 2.24]{HKP}, so not extremely amenable, hence DRP fails. On the other hand, it is easy to see that there is $\eta \in E(S_{\bar c}(\C))$ with a singleton image (cf. \cite[Proposition A.1]{KrRz}), which implies that the Ellis group in $E(S_{\bar c}(\C))$ is trivial. In the KPT context, one can take as $M$ the countable dense circular order. Then $\Aut(M)$ is not extremely amenable, because finite substructures of $M$ are not rigid. On the other hand, using \cite[Theorem 5]{The0}, it is easy to compute that the universal minimal $\Aut(M)$-flow is just the natural action of $\Aut(M)$ on the space $\widehat{M}$ of all cuts of $M$ equipped with the $0$-dimensional topology taken from the space of liner orders on $M$ compatible with the circular order. A similar argument to the one for theories yields an element of the Ellis semigroup of $\widehat{M}$ with singleton image, so the Ellis group of the flow $\widehat{M}$ is trivial. Finally, this last trivial Ellis group coincides with the Ellis group of the universal $\Aut(M)$-ambit by \cite[Lemma 5.16]{CGK}.  

Some questions on relationships between the various  versions of finite Ramsey degrees already appeared in earlier sections, see  Questions \ref{question: sep. fin. EDERdeg for beta G}, \ref{question: sep. fin DERdeg vs sep. fin. EDERdeg}, and Problem \ref{problem: sep. fin EDERdeg and sep. fin ERdeg}. Among other problems, let us mention a strengthening of Question  \ref{question: sep. fin DERdeg vs sep. fin. EDERdeg}: Does DRP imply sep.\ fin.\ EDRdeg? We expect negative answers.

\section*{No AI statement}
No AI was involved in this paper. The project was initiated in 2019, after completing \cite{KLM}. The results were obtained between 2019 and 2025 and discussed by the first author in several conferences and seminars in those years.

\printbibliography

\end{document}